\documentclass{article}

\usepackage{amsmath}
\usepackage{amsfonts}
\usepackage{yfonts}
\usepackage{amssymb}
\usepackage{amsmath,amscd} 
\usepackage{latexsym}
\usepackage{txfonts,pxfonts}
\usepackage[official,right]{eurosym} 
\usepackage{enumerate}
\usepackage{times}

\usepackage[french, english]{babel}
\usepackage[]{aeguill} 

\usepackage{longtable}

\usepackage{array}
\usepackage{multirow} 
\usepackage{graphicx}  
\usepackage{pstricks}
\usepackage[all]{xy}
\usepackage{xypic}
\newcommand{\tab}{\hspace{5mm}}
\newcommand{\tuple}[1]{\ensuremath{\left\langle #1\right\rangle}}

\newcommand{\et}{\wedge}
\newcommand{\ou}{\vee}

\newcommand{\non}{\neg}

\newcommand{\K}{\mathbf{K}} 

\newcommand{\p}{\mathbf{P}}	
\newcommand{\M}{\mathcal{M}} 
\newcommand{\KK}{\mathcal{K}} 

\newcommand{\C}{\mathcal{C}} 
\newcommand{\R}{\mathcal{R}} 

\newcommand{\EL}{\mathsf{EL}}
\newcommand{\CEL}{\mathsf{CEL}}

\newtheorem{theorem}{\bf Theorem}
\newtheorem{example}[theorem]{\bf Example}

\newenvironment{proof}{\begin{trivlist}\item[]{\bf Proof.}}{\hspace*{\fill} 
$\blacksquare$ \end{trivlist}}

\begin{document}

\title{Contextual Epistemic Logic}

\author{\textsc{Manuel Rebuschi}\\ \small{L.H.S.P. -- Archives H. Poincaré}\\ \small{Université Nancy 2}\\ \small{\textsf{manuel.rebuschi@univ-nancy2.fr}} \and \textsc{Franck Lihoreau}\\ \small{Instituto de Filosofia da Linguagem}\\ \small{Universidade Nova de Lisboa}\\ \small{\textsf{franck.lihoreau@fcsh.unl.pt}} }
\date{}

\maketitle

\selectlanguage{english}

\begin{abstract}{One of the highlights of recent informal epistemology is its growing theoretical emphasis upon various notions of context. The present paper addresses the connections between knowledge and context within a formal 
approach. To this end, a ``contextual epistemic logic'', \textsf{CEL}, is proposed, which consists of an extension of standard \textsf{S5} epistemic modal logic with appropriate reduction axioms to deal with an extra contextual operator. 
We describe the axiomatics and supply both a Kripkean and a dialogical semantics for \textsf{CEL}. An illustration of how it may fruitfully be applied to informal epistemological matters is provided.}
\end{abstract}


\section{Introduction}

The formal approach to knowledge and context that we propose in this  
paper was originally driven not only by formal logical concerns but  
also by more informal epistemological concerns.

In the last two or three decades, indeed, epistemology has seen two  
major ``turns'':
\begin{itemize}
\item a ``new linguistic turn'', as Ludlow \cite{Ludlow05} calls it, through the increased reliance, in  
contemporary epistemological debates, upon ``evidence'' regarding how  
we ordinarily \emph{talk} about knowledge, most notably as a result of the  
flourishing discussions about the purported epistemological role of  
various notions of context and the relative merits of  
``contextualism'' 
over ``scepticism'', ``anti-scepticism'', and ``subjectivism'', \emph{inter alia};
\item a ``logical turn'', through the rising conviction that  
discussions in informal epistemology may benefit from formal  
epistemology -- epistemic logic, formal learning theory, belief  
revision, and so on --, notably by applying the methods of epistemic  
logic in order to gain insights into traditional informal  
epistemological issues. Recent work by van Benthem  
\cite{vanBenthem06}, Hendricks \cite{Hendricks04, Hendricks06a} and  
Hendricks \& Pritchard \cite{Hendricks06b}, and Stalnaker \cite{Stalnaker06} count  
as representative of this trend.
\end{itemize}

An important part of the background for the present paper consists of  
the project of taking advantage of the logical turn in order to record  
some of the main lessons to be drawn from the linguistic turn of 
epistemology, the acknowledgement of the possible epistemological role  
of context to start with.

This paper, at the junction of informal and formal epistemology, focuses on the question how to introduce the notion of context into epistemic logic.
It provides a new logic, contextual epistemic logic (\textsf{CEL}), an extension of standard \textsf{S5} epistemic modal logic with appropriate reduction axioms to deal with an extra contextual operator.

In section \ref{contextepist} contemporary informal epistemological discussions over context are briefly exposed, as well as three strategies available to handle context in a formal way. The authors already studied one of them in \cite{LihoreauRebuschi07} and the other two seem quite natural, one of which is accounted for in the present paper. The section ends with a presentation of public announcement logic (\textsf{PAL}), which is technically very close  to \textsf{CEL} yet very different from it in spirit: after a comparison of the respective objectives of the two logics we mention what must be changed in \textsf{PAL} to reach a logic for contexts.

In section \ref{cel} the syntax and Kripke semantics of \textsf{CEL} are given, and a complete proof system. In the subsequent section, we present a dialogical version of \textsf{CEL}: after a general presentation of dialogical logic for (multi-)\textsf{S5} modal systems, we add context relativization and prove its completeness. 

Finally, in section \ref{applications} we provide a connection between our formal definitions and standard epistemological positions (scepticism and anti-scepticism, contextualism, subjectivism), and we show a few applications of \textsf{CEL} to specifically epistemological questions.

\section{Contextual Epistemology \label{contextepist}}

\subsection{From an informal approach to knowledge and context...}

Constitutive of the linguistic turn in epistemology is contextualism.  
Contextualism can be viewed as an attempt to reconcile scepticism and  
anti-scepticism, and subjectivism as an alternative to contextualism.  
All these positions are better thought of as positions about what a  
satisfactory account for the truth of knowledge sentences like ``$i$  
knows that $\varphi$'' or ``$i$ does not know that $\varphi$'' must  
look like. A way of representing what it takes for a knowledge  
sentence to be true is by means of an ``epistemic relevance set'',  
i.e. a set of ``epistemically relevant counter-possibilities'' that an  
agent must be able to rule out, given the evidence available to him,  
for it to be true that he knows a proposition.

Thus, according to \textit{scepticism} -- also called ``sceptical invariantism'', and defended by Unger \cite{Unger71} for instance -- the epistemic standards are so  
extraordinarily stringent that for any (contingent) proposition, all  
logical counter-possibilities to this proposition will count as  
epistemically relevant. That is, for it to be true that an agent $i$
knows a proposition $\varphi$, $i$ must rule out
the entire set of
logical possibilities in which not-$\varphi$;
this being an
unfulfillable task for any $i$ and $\varphi$, it
is never true that
anyone knows anything except, perhaps, necessary
truths.\footnote{One might want to be a sceptic
as far as contingent truths are concerned while
not being a sceptic with respect to necessary
truths. Such truths include logical validities
as well as
analytic truths. Taking the latter into account
within a modal-logical framework can be done by
adopting meaning postulates.}
On the contrary, for \textit{anti-scepticism} -- also called ``Moorean invariantism'' and defended by Austin \cite{Austin46} and Moore \cite{Moore62} -- 
the  
standards for the truth of ``$i$ knows that $\varphi$'' being those  
for an ordinarily correct utterance of this sentence, they are lax  
enough to make (a possibly important) part of our ordinary knowledge  
claims turn out true: for any proposition, not all logical  
counter-possibilities to this proposition will count as epistemically  
relevant. That is, it will be true that $i$ knows that $\varphi$ only  
if $i$ rules out a given (proper) subset of possibilities in which  
not-$\varphi$.

Scepticism and anti-scepticism are \textit{absolutist} views  
about knowledge in the sense that for both of them, whether a logical  
counter-possibility to a proposition is epistemically relevant or not  
depends only on what proposition it is that is purportedly known and  
once this proposition is fixed, the associated epistemic relevance set  
is not liable to vary.
According to \textit{relativist} views, whether a logical  
counter-possibility to a proposition is epistemically relevant or not  
does not depend only on what proposition it is; it also depends on the  
context, for some sense of ``context''. 

For instance, according to  
\emph{contextualism} -- defended by authors like Cohen \cite{Cohen00}, De Rose \cite{DeRose95}, Lewis \cite{Lewis96} --
 the context in question will be that of the  
possible ``knowledge ascriber(s)''. Although we'd better talk of  
contextualism in the plural, a common contextualist assumption is that  
whilst the circumstances of the purported knower are fixed, the  
epistemic standards, therefore the epistemic relevance set, may vary  
with features of the context of the knowledge ascriber(s) such as what  
counter-possibilities they attend to, what propositions they  
presuppose, what is at stake in their own context, etc.
In contrast, \emph{subjectivism} -- also called ``subject-sensitive invariantism'' or ``sensitive moderate invariantism'' -- is defended by authors like Hawthorne \cite{Hawthorne04} and Stanley \cite{Stanley05}. This view
 has it that such factors -- viz.  
attention, interests, presuppositions, stakes, etc. -- are not the  
attention, interests, presuppositions, stakes, etc., of the knowledge  
ascriber, but those of the ``knower'' himself. That is, the epistemic  
standards, therefore the epistemic relevance set, may vary with the  
context of the purported ``knower'', even if no change occurs in the  
circumstances he happens to be in.\footnote{Absent from the foregoing discussion will be another prominent view on the matter, viz. what may be coined ``circumstancialism'' -- the view, held by Dretske \cite{Dretske70, Dretske81} and Nozick \cite{Nozick81}, that the relevance set depends on the objective situation the subject happens to be in. The reason for not discussing this view here is that one constraint on the formal developments in this paper was to stick to epistemological views whose associated epistemic logics (1) were normal, i.e. incorporating at least the principles of the smallest normal modal logic \textsf{K} -- in particular the \emph{Distribution Axiom} and the \emph{Knowledge Generalization Rule}  -- and thus (2) fully characterized some class of standard Kripke models. This is not the case with circumstancialism which drops the two \textsf{K}-principles just mentioned. However, capturing circumstancialist epistemic logic semantically can be done in several ways compatible with a modal approach, two of the most straightforward being (1) by adding ``awareness functions'' into the standard possible worlds models as in \cite{Fagin88}, (2) or by augmenting such models with ``impossible possible worlds'' as in \cite{Rantala82}. Another important epistemological view is the ``assessment sensitivity'' (sometimes called ``relativist'') view, held by MacFarlane \cite{MacFarlane05} and Richard \cite{Richard04} for instance, according to which the relevance set depends neither on the subject's situation or context nor on the attributor's context, but on the context in which a knowledge claim made by an attributor about a subject is being assessed for truth or falsity. For the same reason as before it will not be discussed here, since formalizing it would require reliance on a semantics allowing formulas to be evaluated not only relative to a world, but also to both a context of attribution and a context of assessment -- i.e. a non-standard semantics, perhaps in the vicinity of \cite{Egré06}'s ``token semantics''. These issues will not be pursued here.}

So, depending on whether one opts for an absolutist or a  
relativist view about knowledge, one will assume that a notion of  
context has an epistemological role to play or not. A first incursion  
into formalizing the possible connections between knowledge and context 
and the four epistemological views\footnote{It may appear to some that scepticism does not seem to be a  
respectable language game with a special knowledge operator, but a  
defective language game with a respectable knowledge operator. This  
seems intuitively correct, of course. So why bother trying to capture  
scepticism in a formal framework anyway? Here are two  
reasons. First, the intuition in question might simply be an effect of  
a prejudice against scepticism and thus cannot be appealed to in  
support of a depreciative view of scepticism without begging the  
question; second, our committment is to find a formal framework that  
makes everything ``respectable'' in each of the epistemological  
positions it is meant to capture: we want both their respective  
``language game'' and ``knowledge operator'' to be treated as  
``respectable''.} was undertaken in  
\cite{LihoreauRebuschi07}.


\subsection{... to formal approaches to reasoning about knowledge and context}

The role of context in the above epistemological discussion
is to impose a restriction of the relevant set of possible worlds -- not to go from a possible world to another one. Contexts in a modal formalization  thus cannot simply be reduced to standard modalities. 

As a consequence, in order to represent contexts in epistemic logic three ways seem available:
\begin{enumerate}
	\item use non-standard models, i.e. put the context in the metalanguage and evaluate each (standard epistemic) formula relatively to some world and context;
\vspace{-0,3cm}
	\item use standard models with standard modality, within an extension from basic modal logic;
\vspace{-0,3cm}
	\item use standard models with a non-standard interpretation for context modalities.
\end{enumerate}
The first option, investigated in \cite{LihoreauRebuschi07}, requires that the notion of context enter the metalanguage and that the formulas of the object-language be interpreted in ``contextual models''.
The second track would require a kind of hybrid logic to account for the intersection of two accessibility relations. In the remainder of the subsection we briefly discuss the first two approaches before focusing on the third strategy.

\subsubsection{Contextual models}
Formalizing contexts while sticking to the standard syntax of epistemic logic implies the adoption of non-standard models. Contextual models are usual Kripke models for multi-S5 epistemic languages $\M=\tuple{W,\{\KK_i\}_{i\in I},V}$
augmented with a pair $\left\langle \C, \R\right\rangle$, where:
$\C = \left\{c_j: j \in J\right\} \neq \varnothing$ is a set of  
``contexts'', for $I \subseteq J$,\footnote{A context $c_i$ is thus assigned to every agent $i\in I$, but additional contexts may be added for groups of agents, conversations, and so forth.} 
and
$\R: \C \rightarrow (W\rightarrow\wp(W))$ is a function of ``contextual  
relevance'' associating with each context $c_j$, for each world $w$, the set $\R(c_j)(w)$ of worlds relevant in that context at that world.

Truth in contextual models is then defined relatively to a world and a context. The clauses are the usual ones for propositional connectives in the sense that contexts do not play any role on them,
whereas contexts can modify the evaluation of epistemic operators in one of the following four ways:
\begin{center}
\begin{small}
\begin{tabular}{rl}
$\M, c_i, w \vDash \K_j\varphi$ iff \\

(1.1) & for all $w'$, if  
$\KK_jww'$ and $w'\in\R(c_i)(w)$, then $\M, c_i, w' \vDash \varphi$.\\

(1.2)&  for all $w'$, if $\KK_jww'$ and $w'\in\R(c_i)(w)$, then $\M, c_j, w' \vDash \varphi$.\\

(2.1)&for all $w'$, if $\KK_jww'$ and $w'\in\R(c_j)(w)$, then $\M, c_i, w' \vDash \varphi$.\\

(2.2)&for all $w'$, if $\KK_jww'$ and $w'\in\R(c_j)(w)$, then $\M, c_j, w' \vDash \varphi$.\\
\end{tabular}
\end{small}
\end{center}
The definitions lead to four different logics, i.e. to four interpretations of the epistemic operator $\K_j$. 

Taking epistemic accessibility relations to be equivalence relations,  
a formal characterization of the aforementioned  
informal views of knowledge can be given through a proper choice among the four  
cases (1.1)--(2.2) or/and through proper restrictions on the contextual  
relevance function. The two absolutist views of knowledge --  
scepticism and anti-scepticism -- can thus be associated with case (1.1),
differing from each other only in their respective restrictions on  
$\R$, viz. the restriction that for all $i$ and $w$, $\R(c_i)(w) = W$  
for scepticism -- hence a restriction enabling scepticism to drop the  condition $w'\in\R(c_i)(w)$ from the truth conditions of epistemic sentences -- and the minimal restriction that for all $i$ and $w$,  
$\R(c_i)(w) \subsetneq W$ for anti-scepticism. In contrast, the two  
relativist views of knowledge can be associated with the (--.2)  
cases, viz. case (1.2) for contextualism and case (2.2) for  
subjectivism.\footnote{A possible worry with these formal definitions is that they end up  
making scepticism look like a limiting case of anti-scepticism and the  
characterisation of the latter unsatisfactory for anyone interested in  
a more substantial definition. But first, what we were after was  
more to contrast absolutist with relativist views than to contrast  
scepticism with anti-scepticism. Second, the definition we gave of  
anti-scepticism -- viz. ``for all $i$ and \emph{w}, $\R(c_i)(w)  
\subsetneq W$'' -- was meant to be broad enough so as to allow for  
more specific characterisations of one's favourite version of  
anti-scepticism. This can be done, for instance, by listing all the  
worlds one thinks should be considered epistemically relevant, that  
is, included in $\R(c_i)(w)$, or by providing a general criterion for  
telling the worlds that are epistemically relevant from those that  
aren't (in terms of ``the closest'' or ``close enough'' possible  
worlds, for instance).\label{scepticismdef}}

Within this framework one can express \emph{truth-in-a-context} for  knowledge claims -- an interesting result for contemporary informal  epistemology whose interest in the connection between knowledge and  context lies primarily in the issue of whether the \emph{truth} of knowledge sentences should be taken as absolute or relative. 

Moreover, with contextual models one can account for the logical behaviour of agents of a given  
epistemological type (sceptical, anti-sceptical, contextualist, or subjectivist) reasoning about the knowledge of other agents regardless  
of their epistemological type; that is, one can ask such things as \emph{If contextualism is assumed, then if agent 1 knows that agent 2 knows that $\varphi$, does agent 1 know that $\varphi$?} 
One can also ask such things as \emph{If a contextualist knows that a  subjectivist knows that $\varphi$, does the contextualist know that  $\varphi$?} \\

However, despite its merits, the contextual models framework is
unlikely to win unanimous support from philosophers of language. For
consider satisfaction for epistemic formulas of the form
$\mathbf{K}_j\varphi$ when epistemic operator $\mathbf{K}_j$ is taken
in its (1.2) and (2.2) interpretations. In these cases, the epistemic
operator $\mathbf{K}_j$ manipulates the context parameter against
which the relevant epistemic formula is evaluated: the truth of
$\mathbf{K}_j\varphi$ relative to a context $c_i$ (and world $w$)
depends on the truth of the embedded $\varphi$ relative to
\textit{another context} $c_j$ (and world $w'$).

Now, Kaplan \cite{Kaplan89} famously conjectured that natural language
had no such devices as context-shifting operators -- which he
considered to be ``monsters''. So, if he is right, there can be no
natural language counterpart for $\mathbf{K}_j$ in cases (1.2) and
(2.2). Schlenker \cite{Schlenker03} and others have recently
challenged Kaplan's conjecture, notably by arguing that natural
language allows context-shifting to occur within propositional
attitudes. Nevertheless, the existence of ``monsters'' remains
controversial amongst philosophers of language, and the question of
context-shifting within attitudes is far from having been settled for
that specific kind of attitude reports formed by knowledge claims.

An advantage of the contextual epistemic logic that will be introduced
in the present paper is that it is immune to the charge of
monstrosity: although something resembling context-shifting takes
place in that framework, its language is ``monster-free'' since its
formulas are evaluated against a world parameter only and contexts are
of a purely syntactic nature.

\subsubsection{Hybrid logic}

The semantic intuition lying behind context-relativized attributions of knowledge is that one has to restrict the evaluation to the set of the contextually relevant worlds  to check whether in this restricted set, the agent knows such and such proposition. If one wants to go back to standard Kripke models and wishes to handle contexts as standard (let us say, S5) modalities  $[c_j]$, it is thus required to consider the \textit{intersection} of two accessibility relations: an agent $i$ will be said to know $\varphi$ (at world $w$) relatively to a context $c_j$ 
-- which can be formalized by: $\M,w\vDash\{[c_j]\K_i\}\varphi$ --
if and only if $\M,w'\vDash\varphi$ at every world $w'$ such that $\R_{j,k}ww'$, where $\R_{j,k}$ is the intersection of the two accessibility relations.

Now it is known that the intersection of two relations cannot be defined within basic modal logics \cite{GargovEtAl87}. One has to go beyond basic modal logic to express sound axioms for context-relativized knowledge. Using the hybrid nominals $\nu_x$ and operators $@_x$, we get the natural axiom schema ($\p_i$ being the dual of $\K_i$):\tab
$[c_j]\K_i\varphi \leftrightarrow 
((\p_i\nu_x \et \langle c_k\rangle\nu_x) \to @_x[c_l]\varphi)$, where $\tuple{k,l}\in\{i,j\}\times\{i,j\}$.

Starting from Blackburn's \cite{Blackburn01} dialogical version of hybrid logics one could give a natural account of epistemic logic with standard contextual modalities. This is left for another paper.

\subsubsection{Model-shifting operators}
The way followed in this paper has been paved by recent works about logics of knowledge and communication. The idea, going back to \cite{Plaza89}, is to add \textit{model-shifting operators} to basic epistemic logics.

Van Benthem et al. \cite{vanBenthemetal06} recently accounted for model-shifting operators connected to public announcement. The idea is that the semantic contribution of an announcement, let us say $\varphi$, is to eliminate all those epistemic alternatives falsifying $\varphi$, so that a new model $\M_{|\varphi}$ is obtained relative to which subsequent formulas are evaluated. In this approach, context modalities are clearly not usual modalities.

In the present paper, we propose to use this Public Announcement Logic ($\mathsf{PAL}$) with the required modifications. Before going further into details, let us just add that a frame such as $\mathsf{PAL}$ can be combined with some dynamic account of the context operators $[c_i]$. Van Benthem et al.'s
paper \cite{vanBenthemetal06} provides such a dynamic logic, using standard PDL with an epistemic interpretation. In \cite{LihoreauRebuschi07} the authors of the present paper accounted for contexts using a formalism based on DRT. Anyway, the way contexts are handled (the dynamic part of the formalism) is relatively independent from their role in epistemic attributions.


\subsection{From $\mathsf{PAL}$ to $\mathsf{CEL}$}

\subsubsection{$\mathsf{PAL}$ in a nutshell}

The notion of a public announcement or communication is just that of ``a 
statement made in a conference room in which all agents are present'' 
according to Plaza \cite{Plaza89}, and that of ``an epistemic event where 
all agents are told simultaneously and transparently that a certain formula 
holds right now'' according to van Benthem \emph{et al.} 
\cite{vanBenthemetal06}. The latter intend to model this informal notion by means of a 
modal operator $[\varphi]$, thus allowing for formulas of the form 
$[\varphi]\psi$, read intuitively as ``$\psi$ holds after the 
announcement of $\varphi$.''

The language $\mathcal{L}_{\textsf{PAL}}$ of \textsf{PAL}, the logic of 
public announcement proposed by van Benthem \emph{et al.}, results from 
augmenting the basic epistemic modal language 
$\mathcal{L}$ with operators of the 
form $[\varphi]$ where $\varphi$ is any formula of 
$\mathcal{L}_{\textsf{PAL}}$. The semantics for the resulting language 
$\mathcal{L}_{\textsf{PAL}}$ is based on standard Kripke models for 
$\mathcal{L}$, and the definition of $\vDash$ for 
atoms, negation, conjunction and knowledge operators is as usual with such 
models. The more ``unusual'' feature of the semantics of 
$\mathcal{L}_{\textsf{PAL}}$ is with the clause of satisfaction for 
$[\varphi]$, viz.:
\begin{quote}
\begin{tabular}{lcl}
$\M,w \vDash [\varphi]\psi$ & iff & if $\M,w \vDash \varphi$ then $\M_{|\varphi},w 
\vDash \psi$,
\end{tabular}
\end{quote}
where $\M_{|\varphi}$ is an ``updated model'' consisting of a tuple 
$\left\langle W', \KK'_1, \ldots, \KK'_m, \mathcal{V}'\right\rangle$, with:
\begin{itemize}
\item $W' = \left\{ v\in W: \M,v\vDash \varphi\right\}$,
\vspace{-0,3cm}
\item $\KK'_i = \KK_i \cap (W' \times W')$ for all agents $i \in \left\{1, \ldots, m\right\}$,
\vspace{-0,3cm}
\item $\mathcal{V}'(p) = \mathcal{V}(p) \cap W'$ for all atomic formulas $p$.
\end{itemize}
This appeal to updated models, obtained by restricting a given Kripke model 
to those worlds where a given formula $\varphi$ holds, is meant to capture 
the insight that the epistemic effect of a public announcement of $\varphi$ 
at a time $t$ is that at time $t+1$, all the agents involved 
will have deleted the possible worlds where they do not know that 
$\varphi$.

Our proposed formalism for reasoning about knowledge and context, although 
it is strongly inspired by that for reasoning about public announcements, 
does not appeal to anything like updated models and offers a slightly more 
complicated semantics for the corresponding language.\\

\subsubsection{Different means for different goals} Let us first compare the objectives of the two formalisms. $\mathsf{PAL}$ is specifically designed to account for shifts of common knowledge following public announcements. As a successful annoucement  is made, its content $\varphi$ becomes common knowledge and the model is restricted to all and only those possible worlds where $\varphi$ is true. 

The situation is different with $\mathsf{CEL}$: here sentences are presupposed (and as a consequence, known by each agent) as long as they remain implicit; when a presupposition is made explicit -- asserted, rejected, questioned\ldots -- it is removed from the ``context''. Such a context does not determine the set of epistemic possibilities, but it determines a set of  epistemically relevant possibilities. 

For instance, in usual (non-sceptical) contexts it is assumed that $b$, ``we are all brains in vats'', is false; so $\non b$ belongs to the context. Now if a provocative Sceptic enters the scene and asks whether we are sure that we are not brains in vats, or, better, asserts that we actually are brains in vats, then the context changes: $b$ becomes epistemically relevant. It does not mean that $b$ is true at every accessible world / epistemic possibility, since such a ``public announcement'' is not necessarily approved by all the agents; but $b$ is now a (counterfactual) possibility that needs to be rejected to grant knowledge.\footnote{~As a simplification, we consider only literals as presuppositions; it means that an assertion whose content is a complex proposition will modify the context by removing all the literals occurring in it.
} 

Hence the modifications of contexts according to $\mathsf{CEL}$ do not yield any modification of the accessibility relation: it expands or restricts the whole set of possible worlds. Whereas $\mathsf{PAL}$ is concerned with changes of (common) knowledge of agents, $\mathsf{CEL}$ aims at accounting for changes of knowledge attributions.
The contrast between the two theories is summarized in the following table:

\begin{footnotesize}
\begin{center}
\begin{tabular}{|c|c|}
\hline
$\mathsf{PAL}$ & \textbf{$\mathsf{CEL}$}\\
\hline
Explicit & Implicit\\
Public announcement  & (Proto-)Context \\
$[\varphi]$ & $[c_i]$\\
\hline
Any (complex) formula & Conjunction of literals\\
\hline
What is \textit{a posteriori} known: & What is \textit{a priori} assumed:\\
Set of epistemic possibilities & Set of counterfactual possibilities\\
& (``epistemically relevant possibilities'')\\
\hline
Restriction of the accessibility relation: & Restriction or expansion of the  set of \\
&relevant possible worlds:\\
Removing worlds & Removing or adding worlds\\
\hline
Knowledge shifting & Knowledge attribution shifting\\
\hline
\end{tabular}
\end{center}
\end{footnotesize}
\vspace{0,2cm}

\subsubsection{Consequences} $\mathsf{PAL}$ is useful 
to reach $\mathsf{CEL}$ for it provides a formalization of a logic with a model-shifting operator. However, a few changes are required to adapt $\mathsf{PAL}$ to a logic with contexts:
\begin{itemize}
	\item Model-shifting operators $[\varphi]$ of $\mathsf{PAL}$ are interpreted so that sizes of models are systematically decreasing; in $\mathsf{CEL}$, model-shifting operators $[c_i]$ must allow increasing as well as decreasing sizes of models. 
	
	\item Whereas the announcements of $\mathsf{PAL}$ are obviously linguistic, the contexts of $\mathsf{CEL}$ need not be linguistic in nature; what is required is that they be depicted by some linguistic sentence. 
	
	\item The outcome of a public announcement for the agents' knowledge is straightforward, as is the corresponding axiom schema of $\mathsf{PAL}$. By contrast, several kinds of subtle interactions between context operators and epistemic modalities
	can be grasped in $\mathsf{CEL}$.
	
\end{itemize}
As will be shown in the next section, these slight divergences lead to an important dissimilarity at the
semantic level: $\mathsf{CEL}$ semantics is much more inelegant than $\mathsf{PAL}$ semantics. This will be a good reason enough to go over to the dialogical framework.\\


\section{Contextual Epistemic Logic \label{cel}}

In this section we first introduce the syntax and semantics of 
$\mathsf{CEL}$, then a sound and complete proof system. As was explained above, $\mathsf{CEL}$ is strongly inspired by the model of $\mathsf{PAL}$.

\subsection{$\mathsf{CEL}$: Syntax and Kripke semantics}

\subsubsection{Standard epistemic logic} 
The standard formal approach to reasoning about \emph{knowledge}  
starts with the choice of a basic epistemic modal language. This  
language $\mathcal{L}^m_{\textbf{K}}(\mathcal{A}t)$ consists of the  
set of formulas over a finite or infinite set $\mathcal{A}t = \left\{p_0, p_1, \ldots\right\}$ of atomic formulas  and a set $J = \left\{1, \ldots,  
m\right\}$ of $m$ agents, given by the following form ($F$ standing for a formula of $\mathcal{L}^m_{\textbf{K}}(\mathcal{A}t)$):
\begin{quote}
$F ::= \mathcal{A}t\ |\ \neg F\ |\ (F\et F)\  
|\ \K_j(F)$.
\end{quote}

This language is then  
interpreted in Kripke models consisting of tuples $\M = \left\langle  
W, \R_1, \ldots, \R_m, \mathcal{V}\right\rangle$, where:
$W \neq \varnothing$ is a set of ``worlds'';
$\R_j \subseteq W \times W$ is a relation of ``epistemic  
accessibility'', for all $j \in J$;
$\mathcal{V}:\mathcal{A}t \rightarrow \wp(W)$ is a  
``valuation'' mapping each atomic formula onto a set of worlds.\\

\subsubsection{$\mathsf{CEL}$ Syntax} In order to add contexts to basic epistemic languages, one can choose between two equivalent notations (see \cite{vanBenthem00}). For convenience -- especially for the dialogical approach -- we will not use the modal contextual operator prefixing formulas $[c_i]\varphi$ but a syntactic relativization of formulas: $(\varphi)^{c_i}$. 
Moreover, each context $c_i$ is characterized by a conjunction of literals (i.e. atoms and negation of atoms); the ``context formula'' characterizing a context $c_i$  will be referred to by the same symbol, $c_i$.\footnote{Contexts need not be \textit{identified} with context formulas; they can be partial models or incomplete possible worlds for instance.} So ``context formulas'' $c$ can be defined as follows:
\begin{quote}
$c ::= \mathcal{A}t\ |\ \neg \mathcal{A}t\ |\ \top\  
|\ \bot\ |\ c\et c$.
\end{quote}

We thus recapitulate the syntax of epistemic languages with context $\mathcal{L}^m_{\textbf{KC}}(\mathcal{A}t)$ ($G$ standing for a formula of $\mathcal{L}^m_{\textbf{KC}}(\mathcal{A}t)$), built upon a set of contexts $\C = \{c_i\}_{i\in I}$ ($J\subseteq I$, $J$ being the set of agents):
\begin{quote}
$G ::= \mathcal{A}t\ |\ \neg G\ |\ (G\et G)\  
|\ \K_j(G)\ |\ (G)^{c_i}$.
\end{quote}
Formulas containing at least one operator $\K_j$ are called \textit{epistemic formulas}; formulas containing at least one subformula of the kind $(\varphi)^{c_i}$ are said \textit{context-relativized}, and formulas being not context-relativized are said \textit{absolute}. As is immediatly seen from the definition, $\mathsf{EL}$ is a syntactic fragment of $\mathsf{CEL}$.\\

\subsubsection{$\mathsf{CEL}$ Kripke semantics}
The contextual epistemic language $\mathcal{L}^m_{\textbf{KC}}(\mathcal{A}t)$  
is interpreted relatively to the same Kripke models as $\mathcal{L}^m_{\textbf{K}}(\mathcal{A}t)$. The satisfaction of atoms, absolute negations and conjunctions is defined as usual. For epistemic operators and contextual relativization, the semantics is a bit more complicated:
\begin{small}
\begin{center}
\begin{tabular}{rcl}
$\M,w \vDash p$&iff& $w\in \mathcal V(p)$ \tab ($p$ being atomic)\\
$\M,w \vDash \non\chi$&iff& $\M,w\nvDash\chi$
\\
$\M,w \vDash \chi\et\xi$&iff& $\M,w\vDash\chi$ and $\M,w\vDash\xi$ \\
$\M,w \vDash \K_j\varphi$&iff& for every $w'\in W$ such that $R_jww'$: $\M,w' \vDash\varphi$\\
$\M,w \vDash (p)^{c_i}$&iff& $\M,w\nvDash c_i\;$ or $\;\M,w \vDash p$ \tab ($p$ being atomic) \\
$\M,w \vDash ((\varphi)^{c_k})^{c_i}$&iff& $\;\M,w \nvDash c_i$ or $\;\M,w \vDash (\varphi)^{c_k}$\\

$\M,w \vDash (\non\varphi)^{c_i}$&iff& $\M,w\nvDash c_i\;$ or $\;\M,w \nvDash (\varphi)^{c_i}$\\

$\M,w \vDash (\varphi\et\psi)^{c_i}$&iff& $\;\M,w \vDash (\varphi)^{c_i}$ and $\;\M,w \vDash (\psi)^{c_i}$\\

$\M,w \vDash (\K_j\varphi)^{c_i}$&iff& $\M,w\nvDash c_x\;$ or $\;\M,w \vDash \K_j(\varphi)^{c_y}$ \\
&&where $\tuple{x,y}$ is chosen in $\{i,j\}\times\{i,j\}$\\
\end{tabular}
\end{center}
\end{small}

There are four different clauses hidden in the last one, depending on the values of $x$ and $y$. Each choice regiments a specific position about the interaction between knowledge and context. We will refer to these four possibilities as (1.1), (1.2), (2.1), and (2.2), sometimes adding an explicit exponent to the knowledge operator -- $\K_j^{1.1}$, $\K_j^{1.2}$, $\K_j^{2.1}$, and $\K_j^{2.2}$ respectively -- with the following understanding:

\begin{small}
\begin{center}
\begin{tabular}{rcl}

$\M,w \vDash (\K^{1.1}_j\varphi)^{c_i}$&iff& $\M,w\nvDash c_i\;$ or $\;\M,w \vDash \K^{1.1}_j(\varphi)^{c_i}$ \\ 

$\M,w \vDash (\K^{1.2}_j\varphi)^{c_i}$&iff& $\M,w\nvDash c_i\;$ or $\;\M,w \vDash \K^{1.2}_j(\varphi)^{c_j}$ \\ 

$\M,w \vDash (\K^{2.1}_j\varphi)^{c_j}$&iff& $\M,w\nvDash c_j\;$ or $\;\M,w \vDash \K^{2.1}_j(\varphi)^{c_i}$ \\ 

$\M,w \vDash (\K^{2.2}_j\varphi)^{c_i}$&iff& $\M,w\nvDash c_j\;$ or $\;\M,w \vDash \K^{2.2}_j(\varphi)^{c_j}$ \\

\end{tabular}
\end{center}
\end{small}
Remarks:

\noindent -- In the notation $\K^{u.v}_j$, the first superscript $u$ corresponds to the contextual condition, and the second one $v$ to the context according to which the evaluation is to be continued; the possible values of $u$ and $v$ are $1$ for the current context, and $2$ for the agent.

\noindent -- It is worth noticing that no restriction of the model $\M$ is required for the evaluation of context-relativized formulas.

\vspace{0,2cm}
According to our definition the context relativization $(\varphi)^{c_i}$ of a non-epistemic formula $\varphi$ is nothing but a notational variant for the conditional $c_i\to\varphi$. Since such a relativization has no further impact on epistemic $\K^{1.1}$ formulas, the notion of ``context'' carried by our formalism will appear to be  non-committing for supporters of absolutist conceptions of knowledge. However, this innocent account of context will turn out to be sufficient to model contextualist and subjectivist epistemologies. With $\CEL$ context friends and context ennemies can thus be put together onto the same neutral field.


\subsection{$\mathsf{CEL}$: Proof system}

\subsubsection{Proof systems} There are several proof systems depending on the interaction between epistemic operators and context. The differences are given through axioms of Contextual knowledge.
Each proof systems for $\mathsf{CEL}$ is that for multi-modal $\mathsf{S5}$ epistemic
logic $\mathsf{EL}$ plus the following schemas:
\begin{small}
\begin{center}
\begin{tabular}{rl}
Atoms & $\vdash (p)^{c_i}\leftrightarrow ({c_i}\to p)$\\
Contextual negation & $\vdash(\non\varphi)^{c_i} \leftrightarrow ({c_i}\to \non(\varphi)^{c_i})$\\
Contextual conjunction& $\vdash(\varphi\et\psi)^{c_i} \leftrightarrow ((\varphi)^{c_i}\et(\psi)^{c_i})$ \\
Context iteration &$\vdash ((\varphi)^{c_k})^{c_i} \leftrightarrow (c_i\to(\varphi)^{c_k})$\\
Contextual Knowledge \footnotesize{($\tuple{x,y}\in\{i,j\}\times\{i,j\}$)}&  $\vdash(\K_{j}\varphi)^{c_i}\leftrightarrow({c_x}\to \K_j(\varphi)^{c_y})$
\end{tabular}
\end{center}
\end{small}
as well as the following rule of inference:
{\small 
\begin{center}
Context generalization \tab From $\vdash \varphi$, infer $\vdash (\varphi)^{c_i}$.\\
\end{center}
}
Like for the semantics of $\CEL$, the schema for Contextual Knowledge can be made more explicit through the following versions:
\begin{small}
\begin{center}
\begin{tabular}{rl}
1.1-Contextual Knowledge &  $\vdash(\K^{1.1}_{j}\varphi)^{c_i}\leftrightarrow({c_i}\to \K^{1.1}_j(\varphi)^{c_i})$\\
1.2-Contextual Knowledge &  $\vdash(\K^{1.2}_{j}\varphi)^{c_i}\leftrightarrow({c_i}\to \K^{1.2}_j(\varphi)^{c_j})$\\
2.1-Contextual Knowledge &  $\vdash(\K^{2.1}_{j}\varphi)^{c_i}\leftrightarrow({c_j}\to \K^{2.1}_j(\varphi)^{c_i})$\\
2.2-Contextual Knowledge &  $\vdash(\K^{2.2}_{j}\varphi)^{c_i}\leftrightarrow({c_j}\to \K^{2.2}_j(\varphi)^{c_j})$\\
\end{tabular}
\end{center}
\end{small}

\subsubsection{Soundness and completeness} As $\mathsf{CEL}$ axioms are reduction axioms, each $\mathsf{CEL}$-formula $\varphi$ can be translated into a standard epistemic formula $\varphi'$ such that for every model $\M$ and world $w$: $\M,w\vDash \varphi$ iff $\M,w\vDash \varphi'$.
Completeness of $\mathsf{CEL}$ axiomatics relative to the set of formulas valid in the
class $\M^{rst}$ of reflexive, symmetric and transitive models is thus inherited from that of usual epistemic logic.

\begin{theorem} For formulas in  
$\mathcal{L}^m_{\textbf{KC}}(\mathcal{A}t)$, \textnormal{\textbf{$\mathsf{CEL}$}}  
is a sound and complete axiomatization w.r.t. $\M^{rst}$. 
\label{completenessCEL}
\end{theorem}

\begin{proof} From any $\mathsf{CEL}$ formula $\varphi\in\mathcal{L}^m_{\textbf{KC}}(\mathcal{A}t)$, one can build a tuple \tuple{\varphi_0,\ldots,\varphi_n} of $\mathsf{CEL}$ formulas and reach a formula $\varphi'\in\mathcal{L}^m_{\textbf{K}}(\mathcal{A}t)$ such that: $\varphi_0=\varphi$, $\varphi_n=\varphi'$, and every formula $(\varphi_m \leftrightarrow\varphi_{m+1})$ ($0\leq m\leq n-1$) is an instantiation of an axiom schema. 
As a consequence:\begin{quote}$\vdash_\mathsf{CEL} \varphi$ 
iff $\vdash_{\mathsf{CEL}}  \varphi_1$
iff \ldots
iff $\vdash_{\mathsf{CEL}}  \varphi_{n-1}$
iff $\vdash_{\mathsf{CEL}}  \varphi'$\end{quote}
Since $\varphi'\in\mathcal{L}^m_{\textbf{K}}(\mathcal{A}t)$,  $\vdash_{\mathsf{CEL}}  \varphi'$ is equivalent to $\vdash_{\mathsf{EL}}\varphi'$. Now, $\mathsf{EL}$ is sound and complete, so:
\begin{quote}$\vdash_{\mathsf{CEL}}  \varphi'$ iff $\M^{rst}\vDash\varphi'$ \tab (i.e. iff $\varphi'$ is valid w.r.t. $\M^{rst}$)\end{quote}
Hence in order to prove soundness and completeness of $\mathsf{CEL}$, it suffices to prove that for each reduction axiom schema: $\vdash_{\mathsf{CEL}}\psi\leftrightarrow\psi'$, one gets: $\M^{rst}\vDash\psi$ iff $\M^{rst}\vDash\psi'$. This obtains immediately by virtue of the definitions.
\end{proof}



\section{Dialogical $\mathsf{CEL}$ \label{dialcel}}

In this section, we present a dialogical version of $\mathsf{CEL}$. We first introduce dialogical logic for (multi-)S5, then we extend it to context-relativized formulas.


\subsection{Dialogical (multi-)modal logic}

In a dialogical game, two players argue about a thesis (a formula): The proponent \textbf{P} defends it against the attacks of the opponent \textbf{O}. 
For any set of game rules $\mathsf{Dial\Sigma}$ associated with some logical theory $\Sigma$, we will use the notation $\mathsf{Dial\Sigma}\VDash \varphi$ to say that there is a winning strategy for the proponent in the dialogical game $G_\Sigma(\varphi)$, 
i.e. if playing according the rules of $\mathsf{Dial\Sigma}$, she can defend the formula $\varphi$ against any attack from the opponent -- owning a winning strategy for a game enables a player to win any play of the game. 
The game rules are defined such that a formula $\varphi$ is valid or logically true in $\Sigma$ ($\vDash_\Sigma \varphi$) if and only if $\mathsf{Dial\Sigma}\VDash \varphi$.

The rules belong to two categories: particle rules and structural rules. In the remainder of the section, we just give and briefly explain the rules for games reaching multi-S5 valid formulas.\footnote{~Dialogical modal logics were first introduced in Rahman \& Rückert's paper \cite{RahmanRuckert99}. Readers interested by some more detailed account of both non-modal and modal logics should refer to that paper. For a presentation of game rules close to the present one, see \cite{Rebuschi05}.}\\

\subsubsection{Worlds numbering}

The thesis of the dialogue is uttered at a given world $w$, as well as the subsequent formulas. This world relativization is explicit in dialogue games: we will use labelled formulas of the kind ``$w:\varphi$'', like in explicit tableau systems. 
For that purpose, we need a system of world numbering that reflects syntactically the accessibily relation. We will use the following principles, inspired from Fitting numbers for mono-modal logic:
\begin{small}
\begin{itemize}
	\item The initial world is numbered 1. The $n$ immediate successors of $w$ according to the agent $j$
	are numbered $wj1, wj2, \ldots, wjn$.
	\item An immediate successor $wju$ of a world $w$ is said to be of rank $+1$ relative to $w$, and $w$ is said to be of rank $-1$ relative to its immediate successors. A successor $wjujv$ of a world $w$ is said to be of rank $+2$ relative to $w$, etc. 
\end{itemize}
\end{small}
So for instance, a play on a thesis ``$1:\K_i\K_j\varphi$'' can reach the following labelled formula: ``$1i1j2:\varphi$.''\\

\subsubsection{Particle Rules} The meaning of each logical constant is given through a particle rule which determines how to attack and defend a formula whose main connective is the constant in question. The set, {\small $\mathbf{EL^mPartRules}$}, of particle rules for disjunction, conjunction, subjunction, negation, and epistemic operators is recapitulated in the following table:


\begin{small}
\begin{center}
		\begin{tabular}{||c|c|c||}
			\hline\hline
			\hspace{3cm} & \hspace{3cm} & \hspace{3cm} \\
			 & \textbf{Attack} & \textbf{Defence} \\
 			\hline\hline
 			$w:\,\varphi\ou\psi$ & $w:\,?$ & $w:\varphi$, or $w:\psi$ \\
 			 &  & \footnotesize{(The defender chooses)} \\
 			\hline
 			$w:\,\varphi\et\psi$ & $w:\,?_L$, or $w:\,?_R$ & $w:\varphi$, or $w:\psi$ \\
 			 &  \footnotesize{(The attacker chooses)} & {\footnotesize(respectively)}\\
 			 \hline
 			$w:\,\varphi\to\psi$ & $w:\,\varphi$ & $w:\,\psi$ \\
 			\hline
 			$w:\,\non\varphi$ & $w:\,\varphi$ & $\otimes$ \\
 			&  & \footnotesize{(No possible defence)} \\
 			\hline\hline
 			$w:\,\K_i \varphi$ & $w:\,?_{\K_i/wiw'}$ & $wiw':\,\varphi$ \\
 			&  \footnotesize{(the attacker chooses} &  \\
 			& {\footnotesize an available world $wiw'$)} & \\
 			\hline
 			$w:\,\p_i \varphi$ & $w:\,?_{\p_i}$ & $wiw':\,\varphi$ \\
 			&  & \footnotesize{(the available world $wiw'$} \\
 			& &{\footnotesize being chosen by the defender)} \\
 			\hline\hline
		\end{tabular}
\end{center}
\end{small}

The idea for disjunction is that the proposition $\varphi \vee \psi$, when asserted (at world $w$) by a player, is challenged by the question ``\textit{Which one?}''; the defender has then to choose one of the disjuncts and to defend it against any new attack. The rule is the same for the conjunction $\varphi\wedge \psi$, except that the choice is now made by the attacker: ``\textit{Give me the left conjunct} ($?_L$)'' or ``\textit{Give me the right one }($?_R$)'', and the defender has to assume the conjunct chosen by his or her challenger. For the conditional $\varphi\to \psi$, the attacker assumes the antecedent $\varphi$ and the defender continues with $\psi$. Negated formulas are attacked by the cancellation of negation, and cannot be defended; the defender in this case can thus only counterattack (if she can).

The particle rules for each epistemic operator $\K_i$ and its dual $\p_i$ ($i\in\{1,\ldots,m\}$) enable the players to change the world. They are defined in a way analogous to conjunction and disjunction respectively, regarding the player (challenger or defender) expected to make the relevant choice.\\

\subsubsection{Structural Rules}
In addition to the particle rules connected to each logical constant, one also needs structural rules to be able to play in such and such a way at the level of the whole game. The first five of them yield games for classical propositional logic, and the last two rules regiment
the modal part of epistemic logic.

\begin{center}
\fbox{
\begin{minipage}{0.9\textwidth}
\begin{center}

\begin{small}

\begin{itemize}

\vspace{0,1cm}

\item (PL-0) \textbf{Starting Rule:} The initial formula (the \textit{thesis} of the dialogical game) is asserted by \textbf{P} at world 1. Moves are numbered and alternatively uttered by \textbf{P} and \textbf{O}.  Each move after the initial utterance is either an attack or a defence.

	\item (PL-1) \textbf{Winning Rule:} Player \textbf{X} wins iff it is \textbf{Y}'s turn to play and \textbf{Y} cannot perform any move.

	\item (PL-2) \textbf{No Delaying Tactics Rule:} Both players can only perform moves that change the situation.

\item (PL-3) \textbf{Formal Rule for Atoms:} \textit{At a given world} \textbf{P} cannot introduce any new atomic formula; new atomic formulas must be stated by \textbf{O} first. Atomic formulas can never be attacked.

	\item (PL-4c) \textbf{Classical Rule:} In any move, each player may attack a complex formula uttered by the other player or defend him/herself against \textit{any attack} (including those that have already been defended).\footnote{This rule can be replaced by the following one to get games for intuitionistic logic:\\ \textbf{Intuitionistic Rule:} In any move, each player may attack a complex formula uttered by the other player or defend him/herself against \textit{the last attack that has not yet been defended}.}

\item[] A world $w$ is said to be \textit{introduced} by a move in a play when $w$ is first mentioned either through an asserted labeled formula ($w:\varphi$), or through a non-assertive attack ($?_{K_i/w}$).

\item (ML-frw) \textbf{Formal Rule for Worlds:} 
\textbf{P} \textit{cannot introduce} a new world; new worlds must be introduced by \textbf{O} first.

\item (ML-S5) \textbf{S5 Rule:} \textbf{P} can choose any (given) world.\footnote{Other structural rules could define other usual modal systems (K, D, T, S4, etc.). See \cite{RahmanRuckert99}.}

\end{itemize}
\end{small}

\end{center}
\end{minipage}
}
\end{center}

Now we can build the set of rules for multi-S5 epistemic logic ($\mathsf{EL}$):
\begin{center}
{\small $\mathsf{DialEL}$} $:=$ {\small $\mathbf{EL^mPartRules}$}\\$\cup$ \{PL-0, PL-1, PL-2, PL-3, PL-4c, ML-frc, ML-S5\}\\
\end{center}

It is assumed that this dialogical system is sound and complete, i.e.:
\begin{center}
	$\mathsf{DialEL}\VDash\varphi$ iff $\vDash_\mathsf{EL}\varphi$.
\end{center}
This is shown using strategic tableaus analogous which are similar to usual semantic tableaus, after a reinterpretation of the players' roles.\footnote{The proofs of soundness and completeness of Dialogical $\EL$ are non-trivial. They are not explicitly given in \cite{RahmanRuckert99}, even though a halfway point is reached there.}\\

\subsubsection{Examples}


\begin{example} Let us consider a substitution instance of the Positive Introspection Property (also known as Axiom \textbf{4}):
$\K_i \phi \to \K_i \K_i \phi$. As our dialogical rules correspond to $\mathsf{S5}$, the proponent is expected to have a winning strategy in the corresponding game.

\begin{small}
\begin{center}
		\begin{tabular}{||l rl r||l rl r||}
			\hline\hline
			\hspace{0.1cm} && \hspace{1,5cm} & \hspace{0.2cm} & \hspace{0.2cm} && \hspace{1,5cm} & \hspace{0.1cm} \\
			 && \textbf{O} & & & \textbf{P} &&  \\
 			\hline\hline
 			 && & & &  $1:$&$\K_i a \to \K_i \K_i a$ & (0) \\
 			\hline
 			 (1) & $1:$&$\K_i a$ & 0 & & $1:$&$\K_i \K_i a$ & (2)\\
 			\hline
 			(3) & $1:$&$?_{\K_i/1i1}$ & 2 & & $1i1:$&$\K_i a$ & (4)\\
 			\hline
 			(5) & $1i1:$&$\,?_{\K_i/1i1i1}$ & 4 & & $1i1i1:$&$a$ & (8)\\
 			\hline
 			(7) & $1i1i1:$&$a$ & & 1 & 1: &$\,?_{\K_i/1i1i1}$ & (6)
 			\\
 			\hline\hline
		\end{tabular}
\\\tab\\\textbf{P} wins the play
\end{center}
\end{small}
The numerals within brackets in the external column indicate the moves and the corresponding arguments (here from (0) to (8)); when a move is an attack, the internal array indicates the argument which is under attack; the corresponding defence is written on the same line, even if the move is made later in the play.

Let us comment this particular play in detail. It starts at move (0) with the utterance of the thesis by \textbf{P} at world 1 (PL-0). The formula is challenged by \textbf{O} at move (1), using the particle rule for implication; at move (2) \textbf{P} immediately defends his initial argument. \textbf{O} then attacks the epistemic formula at move (3), and using the correlated particle rule as well as (ML-frw), he introduces a new world, $1i1$: in dialogical games the opponent is considered as using the best available strategy; he thus jumps from one world to another as much as possible, to prevent the proponent to use his concessions (atomic utterances) at a given world.
At move (4) \textbf{P} defends her formula using (ML-S5). Then in (5), \textbf{O} attacks the epistemic operator introducing once more a new world, $1i1i1$. Now the proponent cannot immediately defend her utterance, because it would lead her to utter an atomic formula ($a$) which has not been previously introduced by \textbf{O} at $1i1i1$. At move (6), \textbf{P} thus counterattacks (1) using the new world introduced by \textbf{O}, asking him to utter $a$ at this world (ML-S5); in (7) \textbf{O} defends himself uttering $a$ at $1i1i1$: the atomic formula is now available for \textbf{P}, who can win the play at move (8). As the opponent could not play better -- actually, he could not play differently than he did in this play --, this play shows that there is a winning strategy for \textbf{P} in the game.{\hspace*{\fill} 
$\blacksquare$}
\end{example}

\begin{example} Now we consider a formula with two epistemic operators:
$\K_i \K_j a \to (\K_i a \et\K_j a)$. Here it is not enough to consider one play: after move (2), the opponent can choose either the left or the right conjunct. Depending on this choice, the remainder of the play will not be the same. So after checking that the proponent has a winning strategy in plays where \textbf{O} chooses the left conjunct, one cannot conclude that she has a winning strategy at all: it must be verified that she can also systematically win against \textbf{O} when he chooses the right conjunct.

\begin{small}
\begin{center}
		\begin{tabular}{||l rl r||l rl r||}
			\hline\hline
			\hspace{0.1cm} && \hspace{1,5cm} & \hspace{0.2cm} & \hspace{0.2cm} && \hspace{1,5cm} & \hspace{0.1cm} \\
			 && \textbf{O} & & & \textbf{P} &&  \\
 			\hline\hline
 			 && & & &  $1:$&$\K_i \K_j a \to (\K_i a \et\K_j a)$ & (0) \\
 			\hline
 			 (1) & $1:$&$\K_i \K_j a$ & 0 & & $1:$&$\K_i a \et\K_j a$ & (2)\\
 			 \hline
 			 (3) & $1:$&$?_L$ & 2 & & $1:$&$\K_i a$ & (4)\\
 			 \hline
 			 (5) & $1:$&$?_{K_i/1i1}$ & 4 & & $1i1:$&$a$ & (10)\\
 			 \hline
 			 (7) & $1i1:$&$K_j a$ & & 1 & $1:$&$?_{K_i/1i1}$ & (6)\\
 			 \hline
 			 (9) & $1i1:$&$a$ & & 7 & $1i1:$&$?_{K_j/1i1}$ & (8)\\
 			\hline\hline
		\end{tabular}
\\\tab\\\textbf{P} wins the play
\end{center}
\end{small}

\begin{small}
\begin{center}
		\begin{tabular}{||l rl r||l rl r||}
			\hline\hline
			\hspace{0.1cm} && \hspace{1,5cm} & \hspace{0.2cm} & \hspace{0.2cm} && \hspace{1,5cm} & \hspace{0.1cm} \\
			 && \textbf{O} & & & \textbf{P} &&  \\
 			\hline\hline
 			 && & & &  $1:$&$\K_i \K_j a \to (\K_i a \et\K_j a)$ & (0) \\
 			\hline
 			 (1) & $1:$&$\K_i \K_j a$ & 0 & & $1:$&$\K_i a \et\K_j a$ & (2)\\
 			 \hline
 			 (3*) & $1:$&$?_R$ & 2 & & $1:$&$\K_j a$ & (4*)\\
 			 \hline
 			 (5*) & $1:$&$?_{K_j/1j1}$ & 4 & & $1j1:$&$a$ & (10*)\\
 			 \hline
 			 (7*) & $1:$&$K_j a$ & & 1 & $1:$&$?_{K_i/1}$ & (6*)\\
 			 \hline
 			 (9*) & $1j1:$&$a$ & & 7 & $1:$&$?_{K_j/1j1}$ & (8*)\\
 			\hline\hline
		\end{tabular}
\\\tab\\\textbf{P} wins the play
\end{center}
\end{small}
As expected, there is a winning strategy for the proponent in each case. The formula is thus proved $\mathsf{EL}$ valid.{\hspace*{\fill} 
$\blacksquare$}
\end{example}


\subsection{Adding contextual relativization}

\subsubsection{Particle rules}

The table below gives the particle rules for context-relativized formulas: the rules follow the reduction axioms in a natural way. This new set of particle rules will be referred to as {\small $\mathbf{CEL^mPartRules}$}.

\begin{small}
\begin{center}
		\begin{tabular}{||c|c|c||}
			\hline\hline
			\hspace{3cm} & \hspace{3cm} & \hspace{3cm} \\
			 & \textbf{Attack} & \textbf{Defence} \\
 			\hline
 			$w:\,(p)^{c_i}$ & $w:\,c_i$ & $w:\,p$\\
 			 \small{($p$ being an atom)}&  &  \\
 			\hline
			$w:\,((\varphi)^{c_k})^{c_i}$ & $w:\,c_i$ & $w:\,(\varphi)^{c_k}$\\
 			\hline
 			$w:\,(\varphi\ou\psi)^{c_i}$ & $w:\,c_i$ & $w:\,(\varphi)^{c_i}\ou(\psi)^{c_i}$ \\
 			  			\hline
 			$w:\,(\varphi\et\psi)^{c_i}$ & $w:\,?_L$, or $w:\,?_R$ & $w:\,(\varphi)^{c_i}$, or $w:\,(\psi)^{c_i}$ \\
 			 &  \small{(The attacker chooses)} & {\small(respectively)}\\
 			\hline
 			$w:\,(\varphi\to\psi)^{c_i}$ & $w:\,c_i$ & $w:\,(\varphi)^{c_i}\to(\psi)^{c_i}$ \\
 				\hline
 			$w:\,(\non\varphi)^{c_i}$ & $w:\,c_i$ & $w:\,\non(\varphi)^{c_i}$ \\
 				\hline
 			$w:\,(\K_j\varphi)^{c_i}$ & $w:\,c_x$ & $w:\,\K_j(\varphi)^{c_y}$ \\
 			\hline\hline
		\end{tabular}
\end{center}
\end{small}
Here again, we can give a more explicit version of the last particle rule depending on the kind of the epistemic operator:

\begin{small}
\begin{center}
		\begin{tabular}{||c|c|c||}
			\hline\hline
			\hspace{3cm} & \hspace{3cm} & \hspace{3cm} \\
			 & \textbf{Attack} & \textbf{Defence} \\
 			\hline
 			$w:\,(\K^{1.1}_j\varphi)^{c_i}$ & $w:\,c_i$ & $w:\,\K^{1.1}_j(\varphi)^{c_i}$ \\
 			\hline
 			$w:\,(\K^{1.2}_j\varphi)^{c_i}$ & $w:\,c_i$ & $w:\,\K^{1.2}_j(\varphi)^{c_j}$ \\
 			\hline
 			$w:\,(\K^{2.1}_j\varphi)^{c_i}$ & $w:\,c_j$ & $w:\,\K^{2.1}_j(\varphi)^{c_i}$ \\
 			\hline
 			$w:\,(\K^{2.2}_j\varphi)^{c_i}$ & $w:\,c_j$ & $w:\,\K^{2.2}_j(\varphi)^{c_j}$ \\
 			\hline\hline
		\end{tabular}
\end{center}
\end{small}

\subsubsection{Structural Rule}

There is only one structural rule to add to the listing of ML: the rule that states which player can introduce a context $c_i$, by asserting its characteristic formula ${c_i}$. The idea is that a context should not be assumed in any formal proof.
\begin{center}
\fbox{
\begin{minipage}{0.9\textwidth}
\begin{center}
\begin{small}
\begin{itemize}

\item (ML-frc) \textbf{Formal Rule for contexts:} \textbf{P} \textit{cannot introduce} a new context $c$ in a given world $w$ by playing $w:\,{c}\,$; new contexts must be introduced by \textbf{O} first.

\end{itemize}
\end{small}
\end{center}
\end{minipage}
}
\end{center}
The intuition behind such a rule is that a context can be any formula, including atomic ones; so the proponent should not have powers regarding contexts she does not already have for atoms.

Now we have the set of rules for $\mathsf{CEL}$:
\begin{center}
{\small $\mathsf{DialCEL}$} $:=$ {\small $\mathsf{DialEL}$} $\cup$ {\small $\mathbf{CEL^mPartRules}$} $\cup$ \{ML-frc\}\\
\end{center}
Soundness and completeness will be handled after the following  examples.\\

\subsubsection{Examples}

In the following two tables, we consider the validity of a contextually modified version of positive introspection, where the consequent is evaluated relative to another context than the antecedent. As easily appears through the games, the upshot depends on the chosen position: the formula is valid according to (2.2), but not valid according to (1.2).

\begin{example}
\label{ex4}
Using explicit exponents for operators, the formula to be played is the following one: 
$(\K_i^{1.2} a)^{c_i} \to (\K_i^{1.2} \K_i^{1.2} a)^{c_j}$.

\begin{small}
\begin{center}
		\begin{tabular}{||l rl r||l rl r||}
			\hline\hline
			\hspace{0.1cm} && \hspace{1,5cm} & \hspace{0.2cm} & \hspace{0.2cm} && \hspace{1,5cm} & \hspace{0.1cm} \\
			\textbf{1.2} && \textbf{O} & & & \textbf{P} &&  \\
 			\hline\hline
 			 && & & &  $1:$&$(\K_i a)^{c_i} \to (\K_i \K_i a)^{c_j}$ & (0) \\
 			\hline
 			 (1) & $1:$&$(\K_i a)^{c_i}$ & 0 & & $1:$&$(\K_i \K_i a)^{c_j}$ & (2)\\
 			\hline
 			(3) & $1:$&$c_j$ & 2 & & $1:$&$\K_i (\K_i a)^{c_i}$ & (4)\\
 			\hline
 			(5) & $1:$&$?_{\K_i/1i1}$ & 4 & & $1i1:$&$(\K_i a)^{c_i}$ & (6)\\
 			\hline
 			(7) & $1i1:$&$c_i$ & 6 & & $1i1:$&$\K_i (a)^{c_i}$ & (8)\\
 			\hline
 			(9) & $1i1:$&$\,?_{\K/1i1i1}$ & 8 & & $1i1i1:$&$(a)^{c_i}$ & (10)\\
 			\hline
 			(11) & $1i1i1:$&$c_i$ & 10 & & &&
 			\\			
 			\hline\hline
		\end{tabular}
\\\tab\\\textbf{O} wins the play
\end{center}
\end{small}
After move (11), \textbf{P} cannot answer $a$ for it has not been yet introduced by \textbf{O} at world $1i1i1$. The only possible solution for \textbf{P} would be to attack (1) to force \textbf{O} to utter $\K_i a$ at world $1$, then to force him to utter $a$ at $1i1i1$. But she cannot, since the opponent never introduced $c_i$ at world $1$. So she loses the play.{\hspace*{\fill} $\blacksquare$}

\end{example}

\begin{example}\label{ex5}
Let us now consider the (2.2) version of the same formula:\\ 
$(\K_i^{2.2} a)^{c_i} \to (\K_i^{2.2} \K_i^{2.2} a)^{c_j}$.

\begin{small}
\begin{center}
		\begin{tabular}{||l rl r||l rl r||}
			\hline\hline
			\hspace{0.1cm} && \hspace{1,5cm} & \hspace{0.2cm} & \hspace{0.2cm} && \hspace{1,5cm} & \hspace{0.1cm} \\
				\textbf{2.2}		 && \textbf{O} & & & \textbf{P} &&  \\
 			\hline\hline
 			 && & & &  $1:$&$(\K_i a)^{c_i} \to (\K_i \K_i a)^{c_j}$ & (0) \\
 			\hline
 			 (1) & $1:$&$(\K_i a)^{c_i}$ & 0 & & $1:$&$(\K_i \K_i a)^{c_j}$ & (2)\\
 			\hline
 			\textbf{(3)} & $1:$&$c_i$ & 2 & & $1:$&$\K_i (\K_i a)^{c_i}$ & (4)\\
 			\hline
 			(5) & $1:$&$?_{\K_i/1i1}$ & 4 & & $1i1:$&$(\K_i a)^{c_i}$ & (6)\\
 			\hline
 			(7) & $1i1:$&$c_i$ & 6 & & $1i1:$&$\K_i (a)^{c_i}$ & (8)\\
 			\hline
 			(9) & $1i1:$&$?_{\K_i/1i1i1}$ & 8 & & $1i1i1:$&$(a)^{c_i}$ & (10)\\
 			\hline
 			(11) & $1i1i1:$&$c_i$ & 10 & & $1i1i1:$&$a$&(18)
 			\\
 			\hline
 			(13) & $1:$&$\K_i(a)^{c_i}$ &  & 1 & $1:$&$c_i$& (12)\\
 			\hline
 			(15) & $1i1i1:$&$(a)^{c_i}$ &  & 13 & $1:$&$?_{\K_i/1i1i1}$ & (14)\\
 			\hline
 			(17) & $1i1i1:$&$a$ &  & 15 & $1i1i1:$&$c_i$ \tab\tab  & (16)\\
 			
 			\hline\hline
		\end{tabular}
\\\tab\\\textbf{P} wins the play
\end{center}
\end{small}
Here the proponent has a winning strategy, thanks to the utterance of $c_i$ by \textbf{O} at world $1$ in the third move.
{\hspace*{\fill} $\blacksquare$}
\end{example}

\subsubsection{Soundness and completeness}

\begin{theorem} 
Assuming that 
\textnormal{\textbf{$\mathsf{DialEL}$}} is sound and complete w.r.t. $\M^{rst}$, \textnormal{\textbf{$\mathsf{DialCEL}$}}  
defines a \emph{sound} and \emph{complete} dialogics w.r.t. $\M^{rst}$, i.e. 
for every $\mathsf{CEL}$ formula $\varphi\in\mathcal{L}^m_{\textbf{KC}}(\mathcal{A}t)$, the following equivalence holds:\tab 
$\mathsf{DialCEL}\VDash\varphi$ iff $\M^{rst}\vDash\varphi$.
\end{theorem}

The proof is analogous to that of Theorem \ref{completenessCEL}: we use the reduction axioms to translate $\mathsf{CEL}$ formulas into $\mathsf{EL}$ formulas, and assuming that dialogical logic  $\mathsf{DialEL}$ is sound and complete for $\mathsf{EL}$, we concentrate on the 
axiom schemas of $\mathsf{CEL}$.

\begin{proof} From any $\mathsf{CEL}$ formula $\varphi\in\mathcal{L}^m_{\textbf{KC}}(\mathcal{A}t)$, one can build a tuple \tuple{\varphi_0,\ldots,\varphi_n} of $\mathsf{CEL}$ formulas and reach a formula $\varphi'\in\mathcal{L}^m_{\textbf{K}}(\mathcal{A}t)$ such that: $\varphi_0=\varphi$, $\varphi_n=\varphi'$, and every formula $(\varphi_m \leftrightarrow\varphi_{m+1})$ ($0\leq m\leq n-1$) is an instantiation of an axiom schema. 
As a consequence:
\begin{quote}$\vdash_\mathsf{CEL} \varphi$ 
iff $\vdash_{\mathsf{CEL}}  \varphi_1$
iff \ldots
iff $\vdash_{\mathsf{CEL}}  \varphi_{n-1}$
iff $\vdash_{\mathsf{CEL}}  \varphi'$\end{quote}
According to Theorem \ref{completenessCEL}, $\mathsf{CEL}$ is complete so:
\begin{quote}$\M^{rst}\vDash \varphi$ 
iff $\M^{rst}\vDash  \varphi_1$
iff \ldots
iff $\M^{rst}\vDash  \varphi_{n-1}$
iff $\M^{rst}\vDash  \varphi'$\end{quote}
Since $\varphi'\in\mathcal{L}^m_{\textbf{K}}(\mathcal{A}t)$, as $\mathsf{DialEL}$ is sound and complete, we have:
\begin{quote}$\mathsf{DialCEL}\VDash  \varphi'$ iff
$\mathsf{DialEL}\VDash  \varphi'$ iff $\M^{rst}\vDash\varphi'$.\end{quote}
Hence in order to prove soundness and completeness of $\mathsf{DialCEL}$, it suffices to prove that for each reduction axiom schema: $\vdash_{\mathsf{CEL}}\psi\leftrightarrow\psi'$, the following holds: $\mathsf{DialCEL}\VDash\psi\to\psi'$ and $\mathsf{DialCEL}\VDash\psi'\to\psi$.\\

Let us consider the axiom for Negation. We thus have two (kinds of) games.

\begin{small}
\begin{center}
		\begin{tabular}{||l rl r||l rl r||}
			\hline\hline
			\hspace{0.1cm} && \hspace{1,5cm} & \hspace{0.2cm} & \hspace{0.2cm} && \hspace{1,5cm} & \hspace{0.1cm} \\
			 && \textbf{O} & & & \textbf{P} &&  \\
 			\hline\hline
 			 && & & &  $1:$&$(\non\varphi)^{c_i} \to (c_i \to\non(\varphi)^{c_i})$ & (0) \\
 			\hline
 			 (1) & $1:$&$(\non\varphi)^{c_i}$ & 0 & & $1:$&$c_i \to\non(\varphi)^{c_i}$ & (2)\\
 			\hline
 			(3) & $1:$&$c_i$ & 2 & & $1:$&$\non(\varphi)^{c_i}$ & (4)\\
 			\hline
 			(5) & $1:$&$(\varphi)^{c_i}$ & 4 & & & $\otimes$& \\
 			\hline
 			(7) & $1:$&$\varphi$ & &5 & $1:$&$c_i $ & (6)\\
 			\hline
 			(9) & $1:$&$\non\varphi$ &  & 1& $1:$&$c_i$ & (8)\\
 			\hline
 			& &$\otimes$ &  & 9 & $1:$&$\varphi$& (10)
 			\\
 			\hline\hline
		\end{tabular}
\end{center}
\end{small}

\begin{small}
\begin{center}
		\begin{tabular}{||l rl r||l rl r||}
			\hline\hline
			\hspace{0.1cm} && \hspace{1,5cm} & \hspace{0.2cm} & \hspace{0.2cm} && \hspace{1,5cm} & \hspace{0.1cm} \\
			 && \textbf{O} & & & \textbf{P} &&  \\
 			\hline\hline
 			 && & & &  $1:$&$(c_i \to\non(\varphi)^{c_i})\to(\non\varphi)^{c_i}$ & (0) \\
 			\hline
 			 (1) & $1:$&$c_i \to\non(\varphi)^{c_i}$ & 0 & & $1:$&$(\non\varphi)^{c_i}$ & (2)\\
 			\hline
 			(3) & $1:$&$c_i$ & 2 & & $1:$&$\non\varphi$ & (4)\\
 			\hline
 			(5) & $1:$&$\varphi$ & 4 & & & $\otimes$& \\
 			\hline
 			(7) & $1:$&$\non(\varphi)^{c_i}$ & &1 & $1:$&$c_i $ & (6)\\
 			\hline
 			&&$\otimes$ &  & 7 & $1:$&$(\varphi)^{c_i}$ & (8)\\
 			\hline
 			(9)& &$c_i$ &  8& & $1:$&$\varphi$& (10)
 			\\
 			\hline\hline
		\end{tabular}
\end{center}
\end{small}
In both plays, the opponent could attack $\varphi$ after move (10), but as he has already uttered the same formula before, and as the proponent can attack the same argument several times,
any strategy deployed by \textbf{O} will be turned back as a winning strategy by \textbf{P}.\\

The other implications obtain similarly from the definitions.
\end{proof}


\section{Epistemological Applications \label{applications}}

\subsection{Epistemological positions formalized}

 
Dialogical \textsf{CEL} provides us with powerful  
tools for gaining insights into the informal epistemological debate  over the context-relativity of knowledge claims presented in section \ref{contextepist}.
To illustrate this, let us first show how the four epistemological  positions alluded to earlier -- scepticism, antiscepticism,  contextualism and subjectivism -- can be captured within our  contextual logico-epistemic framework. 

In section \ref{contextepist} the four epistemological positions were introduced in  terms of the ``ruling out'' of ``epistemically relevant  counter-possibilities'', which it is quite natural to understand in  terms of \textsf{S5} epistemic accessibility relations. For if we take  it that $\R_iww'$ iff agent $i$ cannot tell $w$ from $w'$ from all  the information available to him at $w$, we can read ``$i$ can rule  out $w'$ on the basis of the information he has in $w$'' as ``$w'$ is  not epistemically accessible for $i$ from $w$''. Moreover, the appropriate epistemic accessibility relations must be equivalence relations, for it is quite natural to think that a possible world is  ruled out by a subject as soon as it is not \emph{exactly the same as} the actual world with respect to the totality of the subject's evidence or information, that is, if it differs, were it only in a minimal way, from the actual world in that respect;\footnote{Two worlds may differ
in numerous respects, and yet be exactly the same with respect to the evidence (conceived of in
internalistic terms) at an agent's disposal. For instance, an Evil Genius world would be very different
from what we take our world to be, but as the sceptical argument goes, in it we would have exactly
the same evidence as we have in our world.} and \emph{being exactly the same as} is an equivalence relation.

Now, the four epistemological positions are positions for or against  the relativity of knowledge claims to a given type of context --  relativity to the ``knowledge ascriber's'' or ``attributor's'' context  for contextualism, to the ``knower's'' or ``knowing subject's''  context for subjectivism, and no relativity to any context for  scepticism and anti-scepticism. In order to capture these differences  within our contextual logico-epistemic framework, let us opt for the  following conventions: 

\begin{itemize} 

\item When dealing with a contextual operator $(\cdot)^{c_i}$, we  shall take it that the context $c_i$ stands for the set of formulas  that are being presupposed or taken for granted by the agent or group  of agents $i$. So we shall keep in mind that a context $c_i$ may be  understood as the result or/and background of a conversation between  several agents -- what they all take for granted for the purpose of  their linguistic interactions --, if $i$ stands for a group of such  agents, as well as it may be understood as the result or/and  background of an agent's ``conversation'' with himself -- e.g. what he  takes for granted for the purpose of his current reflections --, if  $i$ is a single agent; 

\item In general, $(\varphi)^{c_i}$ shall be read as ``it follows  from context $c_i$ that $\varphi$''. In particular, if  $(\K_j\varphi)^{c_i}$ can be read as ``it follows from context $c_i$  that $j$ knows that $\varphi$'', it may better be read as ``in context $c_i$, $j$ counts  as knowing that $\varphi$'', that is ``given what is  taken for granted by $i$ \ldots''. Moreover, when we have a formula of the  form $(\K_j\varphi)^{c_i}$, we shall take it that the agent (or group)  $i$ is to be called the ``attributor'' and $c_i$ ``attributor $i$'s  context'', and the agent (or group) $j$ is to be called the  ``subject''. On the other hand, if we need to reflect on a $c_x$,  depending on whether $x = i$ or $x = j$, we shall speak of $c_x$ as of  the ``attributor's context'' or as of the ``subject's context''. 

\end{itemize} 
With these conventions in hand, we can now establish correspondences  between the explicitly exponented knowledge operators we distinguished  and the different epistemological positions. 
\subsubsection{Scepticism and Anti-scepticism}
The two \textit{absolutist} views -- scepticism and anti-scepticism -- can both  be associated with the knowledge operator $\K_j^{1.1}$, whose  behaviour with respect to contextual relativization was fully  characterized by the following reduction axiom in \textsf{CEL}: 

\begin{quotation}
	\begin{tabular}{lll} $\vdash_{\textsf{CEL}} (\K_j^{1.1}\varphi)^{c_i} \leftrightarrow (c_i  \rightarrow\K_j^{1.1}(\varphi)^{c_i})$ \end{tabular} 
\end{quotation}
\noindent This axiom says that for subject $j$ to count as knowing  that $\varphi$ in attributor $i$'s context $c_i$, it must follow from  this context that $j$ knows that $\varphi$ holds in that same context.  So it is always the very same attributor's context against which an  agent will count as knowing something. This is precisely what the two  absolutist views about knowledge claims have in common: epistemic  standards and relevance sets are not contextually variable but  constant matters. The only difference between scepticism and  anti-scepticism is that for the former the epistemic standards are too  stringent to ever be met while for the latter they are lax enough to  be met very often and possibly most of the time. So, the appropriate  translation from the informal sceptic/anti-sceptic talk into the  formal $\K_j^{1.1}$ talk simply consists in substituting the  appropriate contexts -- $c_{scep}$ and $c_{anti}$ respectively -- for $c_i$ in the axiom above,
and in adding the following condition: 
\begin{quotation}
\begin{tabular}{lll} 

$c_{anti} \rightarrow c_{scep}$, 

\end{tabular} 
\end{quotation}
\noindent but not \emph{vice versa}. This is the most minimal way to  capture the idea that while $c_{anti}$ excludes all the far-fetched  possibilities arising from radical sceptical concerns and encapsulates  a rather large set of presuppositions shared by most agents in their  ordinary talk about knowledge, $c_{scep}$ excludes all such  presuppositions and encapsulates a rather large set of far-fetched  possibilities.\footnote{What we said in footnote \ref{scepticismdef}
about the definition of anti-scepticism in the contextual model  
framework holds \textit{mutatis mutandis} in the current framework.  
One can work with the version of anti-scepticism they favor simply by  
specifying what, according to their version, can be considered the  
constant set of epistemically relevant presuppositions, that is, the  
constant set of literals that constitute $c_{anti}$.} 
For instance, the sceptic about contingent truths may  not be a sceptic about necessary truths and thus may take the  proposition that $2 \times 2 = 4$ into his $c_{scep}$, just as the  anti-sceptic has this in his $c_{anti}$; but whilst the anti-sceptic  will also have the proposition that he is not a victim of an Evil  Genius in his $c_{anti}$, the sceptic will not let this into his  $c_{scep}$. Actually, $c_{scep}$ can be identified with $\top$, which means that 
context-relativization according to scepticism is no relativization at all.

\subsubsection{Contextualism and subjectivism}

Contrary to the absolutist views, the two relativist views, according  to which epistemic standards and relevance sets are contextually  variable, fall under different knowledge operators. Contextualism is  the view that the variability in question is a variability according  to the attributor's context, not the subject's. It is thus natural to  associate contextualism with operator $\K_j^{1.2}$, whose behaviour  with respect to contextual relativization was fully characterized by: 
\begin{quotation}
\begin{tabular}{lll} $\vdash_{\textsf{CEL}} (\K_j^{1.2}\varphi)^{c_i} \leftrightarrow (c_i  \rightarrow\K_j^{1.2}(\varphi)^{c_j})$, \end{tabular} 
\end{quotation}
\noindent for this axiom says that for subject $j$ to count as knowing  that $\varphi$ in attributor $i$'s context $c_i$, it must follow from  \emph{this} context that $j$ knows that $\varphi$ holds in $j$'s own  context, that is, in the context whose attributor is subject $j$  himself. So it is not always the same attributor's context against  which something will count as being known. Relativity to the  attributor's context is thus encapsulated in knowledge operator  $K_j^{1.2}$. In contrast, relativity to the subject's context is  encapsulated in knowledge operator $K_j^{2.2}$, since the reduction  axiom characterizing its behaviour when contextualized was: 
\begin{quotation}
\begin{tabular}{lll} $\vdash_{\textsf{CEL}} (\K_j^{2.2}\varphi)^{c_i} \leftrightarrow (c_j  \rightarrow\K_j^{2.2}(\varphi)^{c_j})$, \end{tabular} 
\end{quotation}
\noindent which says that for subject $j$ to count as knowing that  $\varphi$ in attributor $i$'s context $c_i$, it must follow from  subject $j$'s own context that $j$ knows that $\varphi$ holds in the  same context, that is, subject $j$'s context. We can thus naturally  associate the subjectivist view, according to which epistemic  standards and relevance sets shift with the subject's context, with  operator $\K_j^{2.2}$.

\subsubsection{What about the remaining operator?}

One may ask which epistemological view could match the  
$\mathbf{K}^{2.1}$-operator. In our opinion the main interest of this  
operator does not lie so much in its possible correspondence with a  
view to be found in the epistemological literature as in the means it  
would offer us to handle indexicality phenomena, for instance if we  
wanted to incorporate personal pronouns like `I' and `he'. Consider a  
true utterance of `He knows that I am here'. To formally account for  
its truth, we would need a knowledge operator whose logical behavior  
would match the following equivalence:  
$(\mathbf{K}_{\texttt{\textit{he}}}\texttt{\textit{I am  
here}})^{c_{\texttt{\textit{I}}}} \leftrightarrow (  
c_{\texttt{\textit{he}}} \rightarrow  
\mathbf{K}_{\texttt{\textit{he}}}(\texttt{\textit{I am  
here}})^{c_{\texttt{\textit{I}}}} )$, where  
${c_{\texttt{\textit{I}}}}$ is the context of the agent who uses `I'  
and $c_{\texttt{\textit{he}}}$ the context of the agent refered to by  
`he'. Of course, we would not thereby have accounted for even a bit of  
the great complexity of natural language indexicality, and it is not  
our intention to do so at all.

\subsection{Knowledge features uncovered}

This formal translation of the four epistemological positions can now  be exploited within our contextual logico-epistemic framework. In the  remainder of this section we will use dialogical \textsf{CEL} to  illustrate how differences between the four epistemological  positions that have been pointed out in the contemporary philosophical  literature on knowledge can be recovered within our formal framework,  as well as to illustrate how differences that have not been touched  upon in the literature can be discovered through that framework. We  will give four such illustrations. 

\vspace{-0,2cm}
\subsubsection{Normality}

One thing to note is that the three operators $\K_j^{1.1}$,  $\K_j^{1.2}$, $\K_j^{2.2}$, behave in the same manner with respect to  the contextualized version of the K axiom for knowledge, viz.: 

\begin{tabular}{lll} $\vdash_{\textsf{CEL}}(\K_j^{u.v}\varphi \wedge  \K_j^{u.v}(\varphi\rightarrow\psi)) \rightarrow \K_j^{u.v}\psi)^{c_i}$. \end{tabular}

\noindent For any value of $\varphi$, $\psi$, $u$ and $v$, this is a  theorem of \textsf{CEL} and it is valid with respect to the class  $\M^{rst}$ of Kripke models with equivalence accessibility relations.  Although it is a trivial result from a logico-epistemic point of view,  it is not from an epistemological one. One can tell from the  literature that it is important for advocates of the four identified  epistemological positions that knowledge closure under known material  implication holds. This is crucial to both the famous sceptical  argument from ignorance and to the famous Moorean anti-sceptical  response to it. Even those who take it that knowledge claims are  context-relative admit that closure holds while insisting that it  holds only within contexts and not across contexts (see \cite{Lewis96}  for the contextualist case and \cite{Hawthorne04} for the subjectivist  case, for instance). 

\vspace{-0,1cm}
\subsubsection{Factivity (or not)}

A more interesting result is that our framework clearly establishes a  difference between case (1.1) -- absolutists -- and cases (1.2) and  (2.2) -- relativists -- with respect to the contextualized version of  the T axiom -- call it $(T)^c$: 

\begin{tabular}{lcl} $(\K_j^{u.v}\varphi\rightarrow\varphi)^{c_i}$. \end{tabular} 

\noindent For consider the following example of \textsf{CEL}-dialogical games for  $(T)^c$, where $\varphi$ is an epistemic formula:

\begin{example} In what follows, we compare the factivity of contextual knowledge according to 1.1, 1.2 and 2.2. In particular, we will check whether $(\K_j \K_k p \to \K_k p)^{c_i}$ is valid or not.
\vspace{-0,2cm}
\begin{small}
\begin{center}
1.1-version: $(\K^{1.1}_j \K^{1.1}_k p \to \K^{1.1}_k p)^{c_i}$
		\begin{tabular}{||l rl r||l rl r||}
			\hline\hline
			\hspace{0.1cm} && \hspace{1,5cm} & \hspace{0.2cm} & \hspace{0.2cm} && \hspace{1,5cm} & \hspace{0.1cm} \\
			 \textbf{1.1}&& \textbf{O} & & & \textbf{P} &&  \\
 			\hline\hline
 			 && & & &  
 			 $1:$&$(\K_j \K_k p \to \K_k p)^{c_i}$ & (0) \\
 			\hline
 			
 			 (1) & $1:$&${c_i}$ & 0 & 
 			 & $1:$&$(\K_j \K_k p)^{c_i} \to (\K_k p)^{c_i}$ & (2)\\
 			\hline
 			
 			(3) & $1:$&$(\K_j \K_k p)^{c_i}$ & 2 & 
 			& $1:$&$(\K_k p)^{c_i}$ & (4)\\
 			\hline
 			
 			(5) & $1:$&$c_i$ & 4 & 
 			& $1:$&$\K_k (p)^{c_i}$ & (6)\\
 			\hline
 			
 			(7) & $1:$&$?_{\K_k/1k1}$ & 6 & 
 			& $1k1:$&$(p)^{c_i}$ & (8)\\
 			\hline
 			
 			(9) & $1k1:$&$c_i$ & 8 & 
 			& $1k1:$&$p$ & (20)\\
 			\hline
 			
 			(11) & $1:$&$\K_j(\K_k p)^{c_i}$ & & 
 			3 & $1:$&$c_i$ & (10)\\
 			\hline
 			
 			(13) & $1:$&$(\K_k p)^{c_i}$ & & 
 			11 & $1:$&$?_{K_j/1}$ & (12)\\
 			\hline
 			
 			(15) & $1:$&$\K_k (p)^{c_i}$ & & 
 			13 & $1:$&$c_i$ & (14)\\
 			\hline
 			 			
 			(17) & $1k1:$&$(p)^{c_i}$ & & 
 			15 & $1:$&$?_{K_k/1k1}$ & (16)\\
 			\hline

 			(19) & $1k1:$&$p$ & & 
 			17 & $1k1:$&$c_i$ & (18)\\
 			\hline
 			 					
		\hline\hline
		\end{tabular}
\vspace{0,1cm}\\\tab\textbf{P} wins the play
\end{center}
\end{small}

\begin{small}
\begin{center}
1.2-version: $(\K^{1.2}_j \K^{1.2}_k p \to \K^{1.2}_k p)^{c_i}$
		\begin{tabular}{||l rl r||l rl r||}
			\hline\hline
			\hspace{0.1cm} && \hspace{1,5cm} & \hspace{0.2cm} & \hspace{0.2cm} && \hspace{1,5cm} & \hspace{0.1cm} \\
			 \textbf{1.2}&& \textbf{O} & & & \textbf{P} &&  \\
 			\hline\hline
 			 && & & &  
 			 $1:$&$(\K_j \K_k p \to \K_k p)^{c_i}$ & (0) \\
 			\hline
 			
 			 (1) & $1:$&${c_i}$ & 0 & 
 			 & $1:$&$(\K_j \K_k p)^{c_i} \to (\K_k p)^{c_i}$ & (2)\\
 			\hline
 			
 			(3) & $1:$&$(\K_j \K_k p)^{c_i}$ & 2 & 
 			& $1:$&$(\K_k p)^{c_i}$ & (4)\\
 			\hline
 			
 			(5) & $1:$&$c_i$ & 4 & 
 			& $1:$&$\K_k (p)^{c_k}$ & (6)\\
 			\hline
 			
 			(7) & $1:$&$?_{\K_k/1k1}$ & 6 & 
 			& $1k1:$&$(p)^{c_k}$ & (8)\\
 			\hline
 			
 			(9) & $1k1:$&$c_k$ & 8 & 
 			& & & \\
 			\hline
 			
 			(11) & $1:$&$\K_j(\K_k p)^{c_j}$ & & 
 			3 & $1:$&$c_i$ & (10)\\
 			\hline
 			
 			(13) & $1:$&$(\K_k p)^{c_j}$ & & 
 			11 & $1:$&$?_{K_j/1}$ & (12)\\

		\hline\hline
		\end{tabular}
\\\tab\\\textbf{O} wins the play
\end{center}
\end{small}
Short comment: After move (13) \textbf{P} would have to utter ``$1:c_j$''; but this context has not been previously introduced by \textbf{O}, so she cannot.\\

\begin{small}
\begin{center}
2.2-version: $(\K^{2.2}_j \K^{2.2}_k p \to \K^{2.2}_k p)^{c_i}$		\begin{tabular}{||l rl r||l rl r||}
			\hline\hline
			\hspace{0.1cm} && \hspace{1,5cm} & \hspace{0.2cm} & \hspace{0.2cm} && \hspace{1,5cm} & \hspace{0.1cm} \\
			 \textbf{2.2}&& \textbf{O} & & & \textbf{P} &&  \\
 			\hline\hline
 			 && & & &  
 			 $1:$&$(\K_j \K_k p \to \K_k p)^{c_i}$ & (0) \\
 			\hline
 			
 			 (1) & $1:$&${c_i}$ & 0 & 
 			 & $1:$&$(\K_j \K_k p)^{c_i} \to (\K_k p)^{c_i}$ & (2)\\
 			\hline
 			
 			(3) & $1:$&$(\K_j \K_k p)^{c_i}$ & 2 & 
 			& $1:$&$(\K_k p)^{c_i}$ & (4)\\
 			\hline
 			
 			(5) & $1:$&$c_k$ & 4 & 
 			& $1:$&$\K_k (p)^{c_k}$ & (6)\\
 			\hline
 			
 			(7) & $1:$&$?_{\K_k/1k1}$ & 6 & 
 			& $1k1:$&$(p)^{c_k}$ & (8)\\
 			\hline
 			
 			(9) & $1k1:$&$c_k$ & 8 & 
 			& & & \\

		\hline\hline
		\end{tabular}
\\\tab\\\textbf{O} wins the play
\end{center}
\end{small}
Short comment: Here \textbf{P} cannot even attack (3): with (2.2), she would have to utter ``$1:c_j$'', which has not been introduced by \textbf{O}.
{\hspace*{\fill} $\blacksquare$}
\end{example}

What this shows is that absolutist contextual knowledge is always  factive while relativist contextual knowledge is not always factive.  More specifically, $(T)^c$ holds for relativist knowledge operators  $\K_j^{1.2}$ and $\K_j^{2.2}$ when they bear on ``absolute''  non-epistemic formulas, but not in the general case.

A precision is required here: There is a possible loss of factivity for contextualism or subjectivism on absolute formulas, but this would be a trivial one like the loss of factivity possibly occurring in standard multi-modal epistemic logic:  the formula $\K_j p \to p$ is trivially falsified at any world $w^\star$ lying beyond the scope of the accessibility relation $\R_j$  where $p$ is false. Analogous cases in $\CEL$ are formulas like $(\K_j p \to p)^{c_i}$, which are falsified at worlds where $c_i$ is true while $c_j$ and $p$ are false; but of course, $c_i$ is as irrelevant for the agent $i$'s context-relativized knowledge as is the world $w^\star$ for his absolute knowledge. However, the situation is different with epistemic formulas like the ones just evaluated through dialogical games: here the context used to falsify the formulas is perfectly relevant to evaluate the agents' knowledge, as is seen in the Figure below: 
\begin{figure}[htbp]
	\centering
		\includegraphics[width=0.5\textwidth]{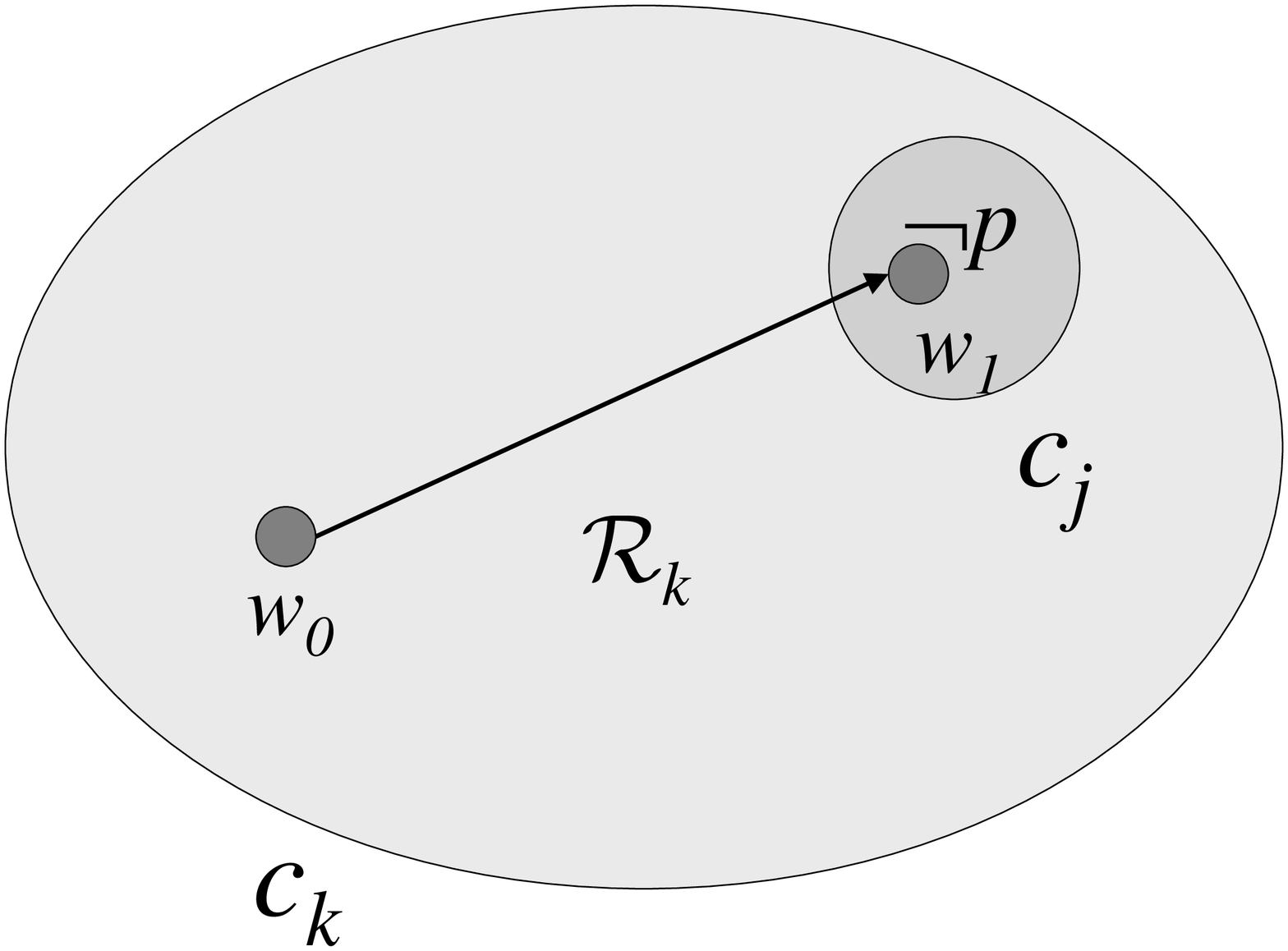}
		\caption{\textit{A counter-model to }$(\K^{\cdot.2}_j \K^{\cdot.2}_k p \to \K^{\cdot.2}_k p)^{c_i}$}
\end{figure}

To our knowledge this result that factivity for contextualism and subjectivism is restricted to knowledge of the world (\emph{versus} of other people's knowledge) has never been highlighted in the epistemological literature\footnote{However, see Stanley \cite{Stanley05} for a discussion of factivity and related matters in a subjectivist setting.}. 

\subsubsection{Context-relativized introspection}

Another first for epistemological discussions over knowledge and  context is the following difference between contextualism and  subjectivism: if subjectivism is true, what one knows in one's own  context, one also knows that one knows it in anyone else's context;  whereas if contextualism is true, what one knows in one's own context,  one may not know that one knows it in anyone else's context. This is  clear from the dialogical games in examples \ref{ex4} and \ref{ex5} in the previous  section, which made it explicit that: 
\begin{itemize} 
\item  $\not\vdash_{\textsf{CEL}}(\K_i^{1.2}p)^{c_i}\rightarrow(\K_i^{1.2}\K_i^{1.2}p)^{c_j}$ 
\item  $\vdash_{\textsf{CEL}}(\K_i^{2.2}p)^{c_i}\rightarrow(\K_i^{2.2}\K_i^{2.2}p)^{c_j}$ 
\end{itemize} 
What this means is that for a subjectivist agent $i$ -- case (2.2) --,  if it follows from the context $c_i$ of which he is the attributor  that he knows that $p$, then it follows from any other attributor  $j$'s context that he knows that $p$; while for a contextualist agent,  this is not true. This can be explained informally as follows: 

\begin{itemize} 

\item (Subjectivist case) If subjectivism is true, then for agent  $i$ to know that $p$, he must meet the standards in vigour in his own  context $c_i$ for knowing that $p$. But he will then \emph{ipso facto}  meet the standards for his knowing that he knows that $p$. For if he  did not know that he knows that $p$, that would be because he  considers it possible that he does not know that $p$; but he cannot  consider this a serious possibility if he already knows that $p$. And  if $i$ counts as knowing that he knows that $p$ in his own context,  and since we are dealing with subjectivism, $i$ will count as knowing  that he knows that $p$ in any other agent $j$'s context. 

\item (Contextualist case) If contextualism is true, then the  relevant standards for knowing a proposition may vary from one  attributor's context to that of another. This being so, agent $i$ may count as  knowing that $p$ relative to the context $c_i$ of which he is the  attributor, whilst not relative to the context $c_j$ of another  attributor $j$ associated with more stringent standards than in $c_i$.  In this case, since one cannot know what is false and since it is  false in $c_j$ that $i$ knows that $p$, it is false in this same $c_j$  that $i$ knows that he knows that $p$. 

\end{itemize} 
It is an advantage of our formal framework that it can capture this  informal difference between contextualism and subjectivism.

\subsubsection{Mixing agents}

Another interesting feature of that framework is that it allows us to  reason about knowledge in a group of epistemologically heterogeneous  agents and to answer formally such informal questions as ``If a  contextualist knows that a subjectivist knows this or that, does the  contextualist know this or that?''\footnote{Suppose that agent $i$ is a $F$-ist (a sceptic, an anti-sceptic, a contextualist, or a subjectivist). If $i$ is coherent with his own theory of knowledge (ascriptions), which he intends to be the only correct one for \emph{any} agent's knowledge, he ought to reason accordingly not only about his own knowledge but also about other agent's knowledge. This amounts to saying \emph{both} that his reasoning about knowledge ought to meet the subjectivist expectations and that he ought to expect other people to meet the same expectations when they reason about knowledge. That is our motivation for talking about sceptical, anti-sceptical, contextualist, and subjectivist \emph{agents} and for asking what they can know about what other types of agents know.} Here we give only one example we  find interesting of that feature through the following question: if an  absolutist agent knows that a subjectivist agent knows a proposition  relative to a context, does the subjectivist know that proposition  relative to that context? In our framework this question becomes that  of deciding whether the following formula is a theorem of \textsf{CEL}: 
\begin{quotation}
\begin{tabular}{lcl} $(\K_j^{1.1}\K_k^{2.2}p)^{c_i} \rightarrow (\K_k^{2.2}p)^{c_i}$. \end{tabular} 
\end{quotation}
\noindent Now, this question is easily settled by means of the  following \textsf{CEL}-dialogical game: 

\begin{example} In the following example, we consider simultaneously agents of different kinds. The upshot is a kind of failure of Axiom T for an absolutist knower.
\begin{small}
\begin{center}
		\begin{tabular}{||l rl r||l rl r||}
			\hline\hline
			\hspace{0.1cm} && \hspace{1,5cm} & \hspace{0.2cm} & \hspace{0.2cm} && \hspace{1,5cm} & \hspace{0.1cm} \\
			 && \textbf{O} & & & \textbf{P} &&  \\
 			\hline\hline
 			 && & & &  
 			 $1:$&$(\K_j^{1.1} \K_k^{2.2} p)^{c_i} \to (\K_k^{2.2}p)^{c_i}$ & (0) \\
 			\hline
 			
 			(1) & $1:$&$(\K_j^{1.1} \K_k^{2.2} p)^{c_i}$ & 0 & 
 			& $1:$&$(\K_k^{2.2} p)^{c_i}$ & (2)\\
 			\hline
 			
 			(3) & $1:$&$c_k$ & 2 & 
 			& $1:$&$\K_k^{2.2} (p)^{c_k}$ & (4)\\
 			\hline
 			
 			(5) & $1:$&$?_{\K_k/1k1}$ & 6 & 
 			& $1k1:$&$(p)^{c_k}$ & (8)\\
 			\hline
 			
 			(9) & $1k1:$&$c_k$ & 8 & 
 			& & & \\
 			\hline

		\hline\hline
		\end{tabular}
\\\tab\\\textbf{O} wins the play
\end{center}
\end{small}
Short comment: \textbf{P} cannot go on and attack (1), for it would require that she utter ``$1:c_i$'', which has not been introduced by \textbf{O}.
{\hspace*{\fill} $\blacksquare$}
\end{example}

 \noindent Thus, the formula is not a principle of our contextual  epistemic logic. This comes as an amendment to, or better as a  complement to what we said earlier about the factivity of absolutist  knowledge. For the lesson to be drawn from this is, roughly, that an  absolutist agent (a sceptic, an anti-sceptic) may know something  without this something being true when this something is about what a  subjectivist agent knows. 

Incidentally, we may notice that the following formula, differing from  the previous one in that the subscripted agent is now the same for  each occurrence of a knowledge operator, is not a principle of our  contextual epistemic logic either: 
\begin{quotation}
\begin{tabular}{lcl} $(\K_j^{1.1}\K_j^{2.2}p)^{c_i} \rightarrow (\K_j^{2.2}p)^{c_i}$. \end{tabular} 
\end{quotation}
\noindent Funnily enough, this could be interpreted in terms of  ``epistemically schizo-phrenic'' agents whose knowledge is  compartmented in the sense that they know different things when they  are in a subjectivist or in a sceptical mood from what they know when they are in an  anti-sceptical or in a contextualist mood, or in the sense that the  subjectivist or the sceptical part of them knows different things from  what the anti-sceptical or the contextualist part of them knows. Then  the lesson to be drawn from the result in question would be that what  your sceptic or anti-sceptic compartment knows about your subjectivist  compartment, your subjectivist compartment may not know of itself.

\section{Conclusion}
\enlargethispage*{0,2cm}

Our primary goal in this paper was to investigate the relationships  
between knowledge and context in the formal framework of epistemic 
modal logic. We thus provided an epistemic logic with context relativization, \textsf{CEL}, together with its dialogical semantics, and applied it to epistemological issues.

The subsequent results can be interesting both from a  
logical and from an epistemological point of view. From the former 
point of view, the interesting upshot is that the logic of public announcements can be translated into a logic for context. The interaction between knowledge and context is slightly more subtle than that of knowledge with announcements, but the result is really close to \textsf{PAL}.

From the epistemological point of  
view, this time, the interesting result is that \textsf{CEL} provides  
us with a powerful formal tool not only for capturing informal views  
about knowledge and context, but also for gaining new insights into  
the debates over their possible interconnections, contributing thereby  
to the current research program in formal epistemology, at the  
interface of logic and the theory of knowledge.

Let us just add that two typical ``Rahmanian issues'' came out  from this work. First, the dialogical version of \textsf{PAL} and its application to epistemology constitute a new confirmation of the fruitfulness of dialogical logic as a framework to combine different logics. Second, we considered diverging agents mutually reasoning about their respective knowledge; this strongly echoes Shahid's recent work \cite{Rahman06} about non-normal logics, classical agents reasoning about intuitionistic ones. Of course, all our agents are normal (and even \textsf{S5}) knowers; anyway as epistemologists, some of them appear to be strange, at least.\\


\paragraph{Acknowledgements} The authors wish to thank Berit Brogaard,
Bertram Kienzle, Helge Rückert, Tero Tulenheimo and an anonymous referee for their inspiring comments and suggestions about earlier versions of this paper.

\end{document}